\documentclass[11pt,twoside,a4paper]{amsart}
\usepackage[inner=2cm,outer=2cm,top=2.5cm,headsep=1cm, footskip=1cm,bottom=3cm]{geometry}

\usepackage{amssymb, amscd, amsmath, enumerate,verbatim}
\usepackage{tikz}
\usepackage{color}
\usepackage{graphicx}
\usepackage[utf8]{inputenc}

\input xy
\xyoption{all}

\numberwithin{equation}{section}
\newtheorem{theorem}{Theorem}[section]
\newtheorem{lemma}[theorem]{Lemma}
\newtheorem{proposition}[theorem]{Proposition}
\newtheorem{corollary}[theorem]{Corollary}

\theoremstyle{definition}

\newtheorem{remark}[theorem]{Remark}
\newtheorem{remarks}[theorem]{Remarks}
\newtheorem{example}[theorem]{Example}
\newtheorem{examples}[theorem]{Examples}
\newtheorem{definition}[theorem]{Definition}

\newcommand{\pb}{\ar@{}[dr]|{\mbox{\LARGE{$\lrcorner$}}}} 
\newcommand{\lra}{\longrightarrow}

\usepackage{caption}

\newcommand{\mk}{\mathbf{k}}

\newcommand{\C}{\mathcal{C}}
\newcommand{\D}{\mathcal{D}}

\newcommand{\M}{\mathcal{M}}

\newcommand{\Cdga}[1]{\mathbf{Cdga}(#1)}

\newcommand{\sMod}{\mathbf{\Sigma Mod}}
\newcommand{\sModc}{\mathbf{\Sigma_{>0} Mod}}
\newcommand{\sModu}{\mathbf{\Sigma_{>1} Mod}}

\newcommand{\lMod}{\mathbf{\Lambda Mod}}
\newcommand{\lModc}{\mathbf{\Lambda_{>0} Mod}}
\newcommand{\lModu}{\mathbf{\Lambda_{>1} Mod}}

\newcommand{\Op}{\mathbf{Op}}

\newcommand{\dgVect}{\mathbf{dgVect_k}}
\newcommand{\gVect}{\mathbf{gVect_k}}

\newcommand{\tree}{\mathbf{Tree}}

\newcommand{\dirlim}{\underset{\longrightarrow}{\mathrm{colim}}\,}

\newcommand{\Aut}{\mathrm{Aut}\,}

\newcommand{\im}{\mathrm{im}\,}

\newcommand{\id}{\mathrm{id}}

\newcommand{\Lie}{\mathcal Lie}
\newcommand{\Com}{\mathcal Com}
\newcommand{\Ass}{\mathcal Ass}
\newcommand{\Ger}{\mathcal Ger}
\newcommand{\Mag}{\mathcal Mag}

\newcommand{\under}[1]{\underline{#1}}

\begin{document}

\title{Up-to-homotopy algebras with strict units}

\author{Agust\'{\i} Roig}
\address[A. Roig]{Dept. Matemàtiques \\ Universitat Polit\`{e}cnica de Catalunya, UPC and BGSMath \\ Diagonal 647, 08028 Barce\-lo\-na, Spain. https://orcid.org/0000-0002-4589-8075 }

\thanks{ Roig partially supported by projects MTM2015-69135-P(MINECO/FEDER), 2014SGR634 and 2017SGR932}


\begin{abstract} We prove the existence of minimal models à la Sullivan for operads with nontrivial arity zero. So up-to-homotopy algebras with strict units are just operad algebras over these minimal models. As an application, we give another proof of the formality of the \emph{unitary} $n$-little disks operad over the rationals.
\end{abstract}

\date{\today}
\maketitle
\tableofcontents

\section{Introduction}

\subsection{} In the beginning, in Stasheff's seminal papers \cite{Sta63}, $A_\infty$-spaces (algebras) had points (units) in what was subsequently termed the zero arity of the operad $\Ass$.\footnote{Therefore, in our notation, we should write it as $u\Ass$, or $\Ass_+$: the \emph{unitary} associative operad.} Stasheff called them \emph{degenerations}. They were still present in \cite{May72} and \cite{BoVo73}, for instance. However, after that, points or units generally disappeared and for a while people working with operads assumed as a starting point $P(0) = \emptyset$, in the topological setting, or $P(0) = 0$ in the algebraic one: see for instance \cite{GK94}. This may have been caused because of the problems posed by those points (units), including

\begin{itemize}
  \item[(1)] \cite{Hin03} had to correct his paper \cite{Hin97} about the existence of a model structure in the category of operads of complexes over an arbitrary commutative ring, excluding the arity zero of the operads---or to consider just the case of characteristic zero.
  \item[(2)] \cite{Bur18} explains how the bar construction of a dg associative algebra with unit is homotopy equivalent to the trivial coalgebra, thus destroying the usual bar-cobar construction through which one usually builds minimal models for operads in Koszul duality theory.
  \item[(3)] \cite{Mar96} (see also \cite{MSS02}) constructs minimal models for operads of chain complexes over a field of zero characteristic, carefully excluding operads with non-trivial arity zero. This allows him to implicitly replace the somewhat \lq\lq wild" general free operad $\Gamma(M)$ for the tamer one that we denote by $\Gamma_0(M)$.
\end{itemize}

More recently, the situation changed, and people have turned their efforts to problems involving non-trivial arity zero. In the topological context, we have the works \cite{MT14}, or \cite{FTW18}, for instance. In the algebraic context we can mention \cite{FOOO09a}, \cite{FOOO09b}, \cite{Pos11}, \cite{Lyu11}, \cite{HM12}, \cite{Bur18}\dots And coping with both, \cite{Mur16}, or \cite{Fre17a} and \cite{Fre17b}.

In introducing points (units) back in the theory of up-to-homotopy things, there are two main possibilities: either you consider \emph{strict} ones, as in Stasheff's original papers \cite{Sta63}, or in \cite{May72}, \cite{Fre17a}, \cite{Fre17b}, \cite{FTW18}, \cite{Bur18}, or you consider \emph{up-to-homotopy} ones, or other relaxed versions of them: \cite{BoVo73}, \cite{FOOO09a}, \cite{FOOO09b}, \cite{Pos11}, \cite{Lyu11}, \cite{HM12}, \cite{MT14}\dots Or you can do both: \cite{KS09}.

In this paper, we work in the algebraic and strict part of the subject. The contribution we add to the present panorama is to prove the existence of minimal models à la Sullivan $P_\infty$ for operads $P$ on cochain complexes over a characteristic zero field $\mk$, with non-trivial arity zero in cohomology, $HP(0) = \mk$. In doing so, we extend the works of Markl \cite{Mar96}, \cite{Mar04} (see also \cite{MSS02}) which proved the existence of such models for non-unitary operads, $P(0) = 0$. Our models include the one of \cite{Bur18} for the unitary associative operad $\Ass_+ = u\Ass$. More precisely, our main result says:

\newtheorem*{I1}{\normalfont\bfseries Theorem $\textbf{\ref{existencia+}}$}
\begin{I1}
Every cohomologically unitary $HP(0) = \mk$, cohomologically connected $HP(1) = \mk$ operad with unitary multiplication $P\in \Mag_+ \backslash \Op $ has a Sullivan minimal model $P_\infty \stackrel{\sim}{\longrightarrow} P$. This minimal model is unitary $P_\infty(0) = \mk$, connected $P_\infty(1) = \mk$ and has a unitary multiplication.
\end{I1}

Here, $\Mag_+$ is the \emph{unitary magmatic} operad: we recall its definition in section \ref{Sigmaoperads}. The restriction condition \emph{operad with unitary multiplication} just means that $P$ is an operad together a morphism $\varphi : \Mag_+ \longrightarrow P$. That is, the unit $1 \in P(0)$ acts in fact as a unit; in other words, there is at least one operation $m \in P(2)$ such that $1$ is a unit for the operation, $m \circ_i 1 = \id, i=1,2$. Therefore, the hypothesis of being (cohomologically) unitary for $P$ is not actually an empty condition.

In the non-unitary case, the importance of such minimal models is well known. For instance, they provide a \emph{strictification} of up-to-homotopy algebras, in that for an operad $P$ (with mild hypotheses), up-to-homotopy $P$-algebras are the same as strict, regular $P_\infty$-algebras. We show how $A_\infty$-algebras with strict units are exactly $(\Ass_+)_\infty = su\Ass_\infty$-algebras.

As an application too, we offer another proof of the formality of the \emph{unitary} $n$-little disks operad $\D_{n+}$ over the rationals. This fills the gap in our paper \cite{GNPR05} noticed by Willwacher in his speech at the 2018 Rio International Congress of Mathematicians \cite{Wil18}.

\subsection{} Markl's mimicking of the Sullivan's original algorithm for dg commutative algebras to non-unitary operads relies on the fact that, when restricted to operads which are non-unitary $P(0) = 0$\footnote{In fact, we show how there is only the need to assume \emph{cohomologically} non-unitary operads, $HP(0) = 0$, in his case.} and cohomologically connected $HP(1) = \mk$, their minimal model is a free graded operad $P_\infty = \Gamma (M)$ over a $\Sigma$-module $M$ which is trivial in arities $0$ and $1$, $M(0) = M(1) = 0$. In this situation, the restricted free graded operad $\Gamma (M)$ has tamer behavior than the \lq\lq wild" general one. We call it $\Gamma_0 (M)$.

Now the point is that, if we want to construct the minimal model à la Sullivan for cohomologically unitary and cohomologically connected operads $HP(0) = HP(1) = \mk$, \emph{keeping the units strict}, we can  also assume that the generating module $M$ also has trivial arities $0$ and $1$. This possibility has been recently made feasible thanks to Fresse's $\Lambda$-modules and $\Lambda$-operads, \cite{Fre17a}.

We recall the definitions of $\Lambda$-modules and $\Lambda$-operads in section $2$, but to put it succinctly, we strip out of the operad all the structure carried by the elements of $P(0)$ and add it to the underlying category of $\Sigma$-modules. For instance, the action of a unit $1 \in \mk = P(0)$ on an arbitrary element $\omega \in P(m)$, $\omega \mapsto \omega \circ_i 1 \in P(m-1)$ becomes part of the structure of the underlying module as a \emph{restriction operation} $\delta_i: P(m) \longrightarrow P(m-1)$. The enhanced category of $\Sigma$-modules with these operations is the category of $\Lambda$-modules. The forgetful functor from operads to $\Lambda$-modules has a left adjoint $\Gamma_+$ which is the one we use to build our minimal models for cohomologically unitary operads. Notice that, as a consequence, the $\Lambda$-structure, or which is the same, the action of the units, becomes fixed and is inherited by the free unitary operad $\Gamma_+$. So the units of our minimal models and their algebras are strict: up-to-homotopy units are excluded in this treatment of the subject.

\subsection{} As with our paper \cite{CR19}, a comparison with the minimal models of operads obtained thanks to the \emph{curved} Koszul duality \cite{Bur18}, \cite{HM12} might be in order. Let us point out a slight advantage of our approach: in order to construct the minimal model of an operad $P$ through the Sullivan algorithm, $P$ does \emph{not} need to fulfill any Koszul duality, curved or otherwise, not even to be quadratic. We only need the simpler conditions on its cohomology $HP(0) \in \{0, \mk \}$, $HP(1) = \mk$, and a unitary multiplication in the cohomologically unitary case.

\subsection{} The contents of the paper are as follows. In section two, we recall some general definitions and facts about $\Sigma$ and $\Lambda$ modules and operads. Section three does the same with trees, free operads and the two particular instances of them we use in the present paper.  Section four contains the basic homotopy theory of operads we need: extensions and their cofibrant properties, and homotopies between morphisms of operads. Section five is devoted to the proof of our main results: the existence and uniqueness of minimal models for dg operads in the non-unitary and unitary case. In section six, we prove the formality mentioned result and check different issues raised by our main results.

Besides Fresse's $\Lambda$-structures, the second main technical device which allows us to transfer the constructions from the non-unitary case to the unitary one are a \emph{simplicial-like} structure for operads with unitary multiplication and a \emph{Kan-like condition} that certain families of elements appearing in our operads verify. The simplicial-like structure and Kan-like condition are the origins of the need for a \emph{unitary multiplication} in the operads to which our results apply.

\section{Notations and conventions}

\subsection{} Throughout this paper, $\mk$ denotes a field of zero characteristic.

Except for a brief appearance of the little disks operad at the end of the paper, all of our operads live in two categories:  $\C = \dgVect $, or $\C = \gVect$, the categories or \emph{dg vector spaces} (also, cochain complexes, differential of degree $+1$) and \emph{graded vector spaces}, over $\mk$. If necessary, we will use the notation $\sMod^\C$, and$ \Op^\C$ for the categories of $\Sigma$-modules and operads with coefficients in $\C$; otherwise, we will omit $\C$ everywhere. Alternatively, we will call their objects \emph{dg operads} and \emph{graded operads}, respectively.

We denote by $0$ the initial object of $\C$ and also by  $\mk$ the unit object of the standard tensor product. $1 \in \mk$ denotes the unit of the field $\mk$. $\id$ denotes the identity of an object in any category and also the image of $1 \in \mk$ by the unit morphism of the operad $\eta : \mk \longrightarrow P(1)$.

\subsection{} Let $C \in \C$ be a dg vector space or a graded space. If $c\in C^n$, we say that $c$ has \emph{degree} $n$ and note it as $|c| = n$. A morphism of complexes $\varphi: C \longrightarrow D$ is a \emph{quasi-isomorphism}, \emph{quis} for short, if it induces an isomorphism in cohomology $\varphi_* = H\varphi : HC \longrightarrow HD$. Given a morphism $\varphi: C \longrightarrow D$ of complexes, we denote by $C\varphi$ the \textit{cone of $\varphi$}. This is the cochain complex given by $C\varphi^n=C^{n+1}\oplus D^n$ with differential

$$
\begin{pmatrix}
-\partial_C   &   0  \\
-\varphi      &   \partial_D
\end{pmatrix}  \ .
$$

We will also denote by $ZC\varphi$, $BC\varphi$ and $HC\varphi = H(C,D)$ the graded vector spaces of the \emph{relative cocycles}, \emph{relative coboundaries} and \emph{relative cohomology}, respectively. The morphism $\varphi$ is a quasi-isomorphism if and only if $HC\varphi =0$.

\subsection{$\Sigma$-modules} Let us recall some definitions and notations about operads (see \cite{KrMay95}, \cite{MSS02}, \cite{Fre17a}).

Let $\Sigma$ be the {\it symmetric groupoid\/}, that is, the category whose objects are the  sets $ \under{n} =  \{1,\dots , n \} $ for $n\ge 1$. For $n=0$, we put $\under{0} = \emptyset$, the empty set. As for the morphisms, $ \Sigma (\under{n},\under{n}) =  \Sigma_n$, and $ \Sigma (\under{}, \under{n}) = \emptyset$  if $m\neq n$, where  $\Sigma_n = \Aut\{1,\dots , n \} $ are the symmetric groups. We will also need to consider its full subcategories $\Sigma_{>1} \subset \Sigma_{>0} \subset \Sigma$, without the $\under{0}, \under{1}$ objects, and the $\under{0}$ object, respectively.

The category of contravariant functors from $\Sigma$ to $\mathcal{C}$ is called the category of $\Sigma$-{\it modules\/} ($\Sigma$-sequences in \cite{Fre17a}) and denoted by $\sMod$. We identify its objects with sequences of objects in $\mathcal{C}$, $M=\left(M(l)\right)_{l\geq 0} = (M(0), M(1), \dots , M(l), \dots)$, with a right $\Sigma_l$-action on each $M(l)$. So, every $M(l)$ is a $\mk [\Sigma_l]$-module, or $\Sigma_l$-module for short. If $\omega$ is an element of $M(l)$, $l$ is called the {\it arity\/} of $\omega$. We say that a $\Sigma$-module $E$ is of \emph{homogeneous arity} $p$ if $E(l) = 0$ for $l \neq p$. If $\omega \in M(l)^p$, we say that $\omega$ has \emph{arity-degree} $(l,p)$. If $M$ and $N$ are $\Sigma$-modules, a {\it morphism of $\Sigma$-modules} $f:M \longrightarrow N$ is a sequence of $\Sigma_l$-{equivariant\/} morphisms $f(l): M(l) \longrightarrow N(l), \ l\geq 0$. Such a morphism is called a \emph{quasi-isomorphism} if every $f(l) : M(l) \longrightarrow N(l)$ is a quasi-isomorphism of complexes for all $l\geq 0$.

If $M$ is a $\Sigma$-module, it's clear that cocycles, coboundaries and cohomology, $ZM, BM, HM$, inherit a natural $\Sigma$-module structure too. In the same vein, if $\varphi : M \longrightarrow N$ is a morphism of $\Sigma$-modules, also its cone $C\varphi$ and the relative cocycles, coboundaries and cohomology, $ZC\varphi, BC\varphi, HC\varphi = H(M,N)$ inherit natural $\Sigma$-module structures. Finally, it is equally clear that projections from cocycles to cohomology $\pi : Z \longrightarrow H$ are morphisms of $\Sigma$-modules, with these inherited structures.

We also use the categories $\sModc$ and $\sModu$ of contravariant functors from $\Sigma_{>0}$ and $\Sigma_{>1}$ to $\C$. We can also consider $\sModu$ and $\sModc$ as the full subcategories of $\sMod$ of those $\Sigma$-modules $M$ such that $M(0) = M(1) = 0$ and $M(0) = 0$, respectively.

\begin{remark}\label{projectius} We are going to resort to the fact that over the group algebras $\mk [\Sigma_n] $ all modules are projective. So, for any $\Sigma$-module $M$ and any $n$, $M(n)$ is a projective $\Sigma_n$-module. This is a consequence of Maschke's theorem.
\end{remark}

\subsection{$\Sigma$-operads}\label{Sigmaoperads} $\Op$ denotes the category of $\Sigma$-operads. Operads can be described as $\Sigma$-modules together with either \emph{structure morphisms} \cite{MSS02} (also called \emph{full composition products} \cite{Fre17a}), $\gamma_{l;m_1,\dots ,m_l} : P(l)\otimes P(m_1) \otimes \dots \otimes P(m_l)\longrightarrow P(m)$, or, equivalently, \emph{composition operations} \cite{MSS02} (also called \emph{partial composition products} \cite{Fre17a}), $
\circ_i : P(l) \otimes P(m) \longrightarrow P(l+m-1)$, and a \emph{unit} $\eta : \mk \longrightarrow P(1)$, satisfying equivariance, associativity and unit axioms (see \cite{KrMay95}, \cite{MSS02}, \cite{Fre17a}). If $P$ and $Q$ are operads, a \emph{morphism of operads} $\varphi : P \longrightarrow Q$ is a morphism of $\Sigma$-modules which respects composition products and units. A morphism of operads is called a \emph{quasi-isomorphism} if it is so by forgetting the operad structure.

We call an operad $P \in \Op$:

\begin{itemize}
  \item[(a)] \emph{Non-unitary} if $P(0) = 0 $, and we denote by   $\Op_0$ the subcategory of \emph{non-unitary operads}.
  \item[(b)] \emph{Unitary} if $P(0) = \mk$, and we denote by  $\Op_+$ the subcategory of \emph{unitary operads}. To emphasize, we sometimes denote them by $P_+$. Morphisms $\phi : P \longrightarrow Q$ in $\Op_+$ are required to be the identity in the zero arity, $\phi(0) = \id : P(0) \longrightarrow Q(0)$. We call them \emph{unitary morphisms}.
  \item[(c)] \emph{Connected} if $P(1) = \mk$. In order not to add more unnecessary notation, we still denote by $\Op_0$ and $\Op_+$ the subcategories of $\Op$ of \emph{non-unitary and connected} operads and \emph{unitary and connected operads}, respectively.
  \item[(d)] \emph{With unitary multiplication}, if there is a unitary morphism $\phi: \Mag_+ \longrightarrow P$, and we denote its category by $\Mag_+ \backslash \Op$. Notice that this condition entails the existence of a multiplication operation $m \in P$ such that $m \circ_1 1 = \id = m \circ_2 1$. Below, we recall the definition of the \emph{unitary magmatic operad} $\Mag_+$.
\end{itemize}

Two  basic operations we perform on our operads, when possible, are the following:

\begin{itemize}
  \item[(a)] Let $P$ be a connected operad. Denote by $\overline{P}$ its \emph{augmentation ideal}. It is the $\Sigma$-module

      $$
      \overline{P}(l) =
      \begin{cases}
        0 , & \mbox{if } l =0,1, \\
        P(l), & \mbox{otherwise}.
      \end{cases}
      $$

  \item[(b)] We say that a non-unitary operad $P$ admits a \emph{unitary extension} when we have a unitary operad $P_+$, which agrees with $P$ in arity $l> 0$ and  composition operations extend the composition operations of $P$. In this case, the canonical embedding $i_+ : P \longrightarrow P_+$ is a morphism in the category of operads.
\end{itemize}

Later on, we recall when a non-unitary operad admits such a unitary extension.

Some particular operads that appear in our paper are the following. We are going to describe the using the graded free operad functor $\Gamma$. We recall some stuff about $\Gamma$ in section \ref{free}.

\begin{itemize}
  \item[$\bullet$] The (non-unitary) \emph{magmatic operad} $\Mag$ is the operad of magmas: algebras with just one arity two operation, no relations at all. That is, $\Mag$ is the free graded operad

      $$
      \Mag = \Gamma (\mu, \sigma\cdot \mu) \ ,
      $$

      \noindent where $\mu$ is an arity two operation and $\sigma \in \Sigma_2$ is the permutation $(2 \ 1)$. As a dg operad, we necessarily have $\partial \mu = 0$. The \emph{unitary magmatic operad} $\Mag_+$ is its unitary extension:

      $$
      \Mag_+ = \dfrac{\Gamma (1, \mu, \sigma\cdot\mu)}{\langle \mu \circ_1 1 - \id,\ \mu \circ_2 1 - \id \rangle} \ .
      $$

      \noindent It adds a unit $1 \in \Mag_+(0) = \mk$ to $\Mag$ \emph{and} the relation $\mu_2 \circ_1 1 = \id = \mu_2 \circ_2 1$. As dg operad, we necessarily have $\partial 1 = \partial \mu = 0$.

  \item[$\bullet$] The (non-unitary) \emph{associative operad} $\Ass$ is the operad of associative algebras:

  $$
  \Ass = \dfrac{\Gamma (\mu, \sigma\cdot\mu)}{\langle \mu \circ_1 \mu - \mu \circ_2 \mu \rangle} \ .
  $$

  \noindent As a dg operad, we necessarily have $\partial \mu = 0$. Its unitary extension $\Ass_+$ is the \emph{unitary associative operad}:

  $$
  \Ass_+ = \dfrac{\Gamma (1, \mu, \sigma\cdot\mu)}{\langle \mu \circ_1 \mu - \mu \circ_2 \mu \ ,         \mu \circ_1 1 - \id,\ \mu \circ_2 1 - \id \rangle} \ .
  $$

  \noindent As dg operad, we necessarily have $\partial 1 = \partial \mu = 0$.

  \item[$\bullet$] The (non-unitary) \emph{commutative operad} $\Com$ is the operad of commutative algebras:

  $$
  \Com = \dfrac{\Gamma (\mu)}{\langle \mu \circ_1 \mu - \mu \circ_2 \mu, \ \mu \circ_1 \mu - \mu\circ_2\mu \rangle} \ .
  $$

  \noindent Again, as a dg operad, $\partial \mu = 0$. We are not going to use its unitary extension
\end{itemize}

\subsection{$\Lambda$-modules} Following \cite{Fre17a}, in order to produce minimal models for our unitary operads, we split the units in $P(0)$ out of them. However, we do not want to forget about this arity zero term, so we \lq\lq include" the data of the units in the $\Sigma$-module structure as follows. Let $\Lambda$ denote the category with the same objects as $\Sigma$ but with morphisms $\Lambda (\under{m}, \under{n}) = \left\{  \text{injective  maps}\  \under{m} \longrightarrow \under{n} \right\}$. We also consider its subcategories $\Lambda_{>1} \subset \Lambda_{>0}\subset \Lambda$, defined in the same way as the ones of $\Sigma$. So, a $\Lambda$-structure on $M \in \sMod$ is just a contravariant functorial morphism $u^* : M(n) \longrightarrow M(m)$ for every injective map $u: \underline{m} \longrightarrow \underline{n}$. If $M$ is the $\Lambda$-module associated with a unitary operad $P_+$, then, for every $\omega\in P_+(n)$, we have $u^*(\omega) = \omega (1, \dots, 1, \id, 1 , \dots , 1, \id , 1, \dots, 1)$, with $\id$ placed at the $u(i)$-th variables, for $i = 1, \dots , m$.

All these $u^*$ can be written as compositions of the \emph{restriction operations} that come from some particular injective maps  $u = \delta^i: \underline{n-1} \longrightarrow \underline{n}$:

$$
\delta^i(x) =
\begin{cases}
  x, & \mbox{if } x= 1, \dots , i-1 \\
  x+1, & \mbox{if } x= i, \dots , n-1  \ .
\end{cases}
$$

Again, if $M = P_+$ and we put $\delta_i = (\delta_i^*)$, this means $\delta_i (\omega) = \omega (\id, \dots, \id, 1 , \id , \dots , \id)$, with $1$ in the $i$-th variable. For instance, an augmentation $\varepsilon : M(n) \longrightarrow \mk$, $\varepsilon (\omega) = \omega (1, \dots, 1)$, can be written as a composition like $\varepsilon = \delta_1 \circ \delta_1 \circ \dots $. So whenever we want to define a $\Lambda$-structure on $M$, we can restrict ourselves to define those $\delta_i : M(n) \longrightarrow M(n-1), i = 1, \dots , n, n\geq 1$, subjected to the natural contravariant functorial constraints $(\delta^i\delta^j)^* = \delta_j\delta_i$, equivariant relations..., that we can find in \cite{Fre17a}, p. 71.

$\lMod$ denotes the category of contravariant functors from $\Lambda$ to $\C$ and it is called the category of $\Lambda$-modules ($\Lambda$-sequences in \cite{Fre17a}). We still have the obvious notions of arity, morphisms and the full subcategories $\lModu$ and $\lModc$ for $\Lambda$-modules.

Since restriction operations commute with differentials, cocycles, coboundaries and cohomology, $ZM$, $BM$, $HM$ inherit natural structures of $\Lambda$-modules from a $\Lambda$-module. The same is true if we have a morphism $\varphi : M \longrightarrow N$ of $\Lambda$-modules: its cone $C\varphi$, relative cocycles, coboundaries and cohomology, $ZC\varphi, BC\varphi, HC\varphi = H(M, N)$ inherit natural $\Lambda$-structures. Nevertheless, we are \emph{not} going to use this natural structure, but to \emph{chose} the more convenient one at each step of Sullivan's algorithm: see definition \ref{OpdefKS+}.

\subsection{$\Lambda$-operads} Let $P_+$ be an unitary operad $P_+(0) = \mk$. We can associate to $P_+$ a non-unitary one $P = \tau P_+$, its truncation,

$$
P (l) =
\begin{cases}
  0, & \mbox{if } l = 0, \\
  P_+(l), & \mbox{otherwise}.
\end{cases}
$$

together with the following data:

\begin{itemize}
  \item[(1)] The composition operations $\circ_i : P_+(m)\otimes P_+(n) \longrightarrow P_+(m+n-1)$ of $P_+$ , for $m, n >0$.
  \item[(2)] The restriction operations $u^*: P_+(n) \longrightarrow P_+(m)$, for every $u\in \Lambda (\under{m},\under{n})$, for $m,n >0$. These restrictions are defined as $u^*(\omega) = \omega (1, \dots, 1, \id, 1 , \dots , 1, \id , 1, \dots, 1)$, with $\id$ placed at the $u(i)$-th variables, for $i = 1, \dots , m$.
  \item[(3)] The augmentations $\varepsilon : P_+(m) \longrightarrow \mk = P_+(0)$,   $\varepsilon (\omega) = \omega (1,  \dots , 1)$, for $m>0$.
\end{itemize}

A non-unitary operad $P$, together with the structures $(1), (2), (3)$ is called a $\Lambda$-operad.

\begin{remark} Truncation makes sense also for operads having arbitrary zero arity $P(0)$. See \ref{twominimalmodels}.
\end{remark}

According to \cite{Fre17a}, p. 58, every unitary operad $P_+$ can be recovered from its non-unitary truncation $P$ with the help of these data, which define the category of $\Lambda$-operads, $\Lambda \Op_0$, and its corresponding variants (see \cite{Fre17a}, page 71). That is, we have an  isomorphism of categories

$$
\tau : \Op_{+} =  \mathbf{\Lambda}\Op_0 / \Com: (\ )_+  \ .
$$

Here, $(\ )_+$ denotes the \emph{unitary extension} associated with any non-unitary and augmented $\Lambda$-operad (see \cite{Fre17a}, p. 81). Namely, if $P \in \mathbf{\Lambda}\Op_0 / \Com$, its unitary extension $P_+$ is the $\Sigma$-operad defined by

$$
P_+ (l) =
\begin{cases}
  \mk, & \mbox{if } l = 0, \\
  P(l), & \mbox{otherwise}.
\end{cases}
$$

And the unitary operad structure is recovered as follows:

\begin{itemize}
  \item[(1)] Composition operations $\circ_i : P_+(m)\otimes P_+(n) \longrightarrow P_+(m+n-1)$ for $m, n >0$ are those of $P$.
  \item[(2)] For $n>1$, the \emph{restriction operation} $u^* = \delta_i: P(n) \longrightarrow P(n-1)$ gives us the partial composition operation $\_\circ_i 1: P_+(n)\otimes P_+(0) \longrightarrow P_+(n-1), \ i=1,\dots , n$.
  \item[(3)] The augmentation $\varepsilon : P(1) \longrightarrow \mk$ gives the unique partial composition product $P_+(1)\otimes P_+(0) \longrightarrow P_+(0)$.
\end{itemize}

Let us end this section with a couple of easy remarks.

\begin{lemma}\label{extensionfunctor} The unitary extension functor $(\ )_+$ commutes with cohomology and colimits. That is, $H(P_+) = (HP)_+$ and $\dirlim_n (P_{n})_+ = (\dirlim_n P_n)_+$.
\end{lemma}

\begin{proof} Commutation with cohomology is obvious. Commutation with colimits is a consequence of $(\ )_+$ having a right adjoint, namely the truncation functor $\tau$.
\end{proof}

As a consequence, $(\ )_+$ is an exact functor.

\begin{remark} The initial object of the category of general operads $\Op$ is the operad $I$

$$
I(l) =
\begin{cases}
  \mk , & \mbox{if } l = 1, \\
  0, & \mbox{otherwise} \ ,
\end{cases}
$$

and the obvious operad structure. It is also the initial object of the subcategory of non-unitary connected operads $\Op_0$. We denote it also by $I_0$. Its unitary extension $I_+$

$$
I_+(l) =
\begin{cases}
  \mk , & \mbox{if } l = 0, 1, \\
  0, & \mbox{otherwise} \ ,
\end{cases}
$$

with the only possible non-zero partial composition operation being the identity, is the initial object of the subcategory of unitary and connected operads, $\Op_+$.
\end{remark}

\section{Free operads}

We recall in this section the definition of the general free operad and of two of its particular instances we are going to use. We start with a review of trees. Trees are useful to represent elements (operations) of operads, its composition products, and to produce an accurate description of the free operad.

\subsection{Trees} When we speak of trees, we adhere to the definitions and conventions of \cite{Fre17a}, appendix I. We include a summary here, for the reader's convenience.

\begin{definition} An $r$-\emph{tree} $T$ consists of:

\begin{itemize}
\item[(a)] A finite set of \emph{inputs}, $\underline{r} = \left\{i_1, \dots , i_r \right\}$ and an \emph{output} $0$.
\item[(b)] A set of \emph{vertices} $v \in V(T)$.
\item[(c)] A set of \emph{edges} $e\in E(T)$, oriented from the \emph{source} $s(e) \in V(T) \sqcup \underline{r}$ towards the \emph{target} $t(e) \in V(T) \sqcup \left\{0 \right\}$.
\end{itemize}

These items are subjected to the following conditions:

\begin{itemize}
  \item[(1)] There is one and only one edge $e_0 \in E(T)$, the \emph{outgoing edge of the tree}, such that $t(e_0) = 0$.
  \item[(2)] For each $i \in \underline{r}$, there is one and only one edge $e_i \in E(T)$, the \emph{ingoing edge of the tree indexed by $i$}, such that $s(e_i) = i$.
  \item[(3)] For each vertex $v \in V(T)$, there is one and only one edge $e_v \in E(T)$, \emph{the outgoing edge of the vertex $v$}, such that $s(e_v) = v$.
  \item[(4)] Each $v \in V(T)$ is connected to the output $0$ by a chain of edges $e_v, e_{v_{n-1}}, \dots , e_{v_1}, e_{v_0}$ such that $v = s(e_v), t(e_v)= s(e_{v_{n-1}}), t(e_{v_{n-1}}) = s(e_{v_{n-2}}), \dots , t(e_{v_2}) = s(e_{v_1}), t(e_{v_1}) = s(e_{v_0})$ and $t(e_{v_0}) = 0$.
\end{itemize}
\end{definition}

Some fundamental examples of trees:

\begin{itemize}
  \item[(1)] The $r$-\emph{corolla} $Y_r $ is the only tree having just one vertex, $r$ inputs, and one output.
  \item[(2)] The \emph{unit tree} is the only tree without vertices; just one input and one output.
  \item[(3)] \emph{Corks} are trees without inputs, just one output, and just one vertex.
\end{itemize}

An operation of $r$ variables (an element of arity $r$) $w\in P(r)$ can be depicted as a tree of $r$ inputs. The unit tree represents the identity $\id \in P(1)$ and a cork can be thought as a unit $1 \in P(0)$.

However, we are going only only to consider trees that fulfill the following additional  property

\begin{itemize}
  \item[(5)] For each vertex $v\in V(T)$, we have at least one edge $e\in E(T)$ such that $t(e) = v$,
\end{itemize}

In other words, except for the unit $1 \in \mk = P(0)$, our trees don't have real \emph{corks}, because the only time composition operations involve grafting a cork onto a tree, we apply the reduction process described in \cite{Fre17a}, A.1.11, consisting of re-indexing the inputs and removing the cork where we apply it.

To a vertex $v$ we also associate a set of \emph{ingoing edges}: those edges whose target is $v$. Let's denote its cardinal by $r_v = \sharp \left\{ e\in E(T)  \ | \ t(e) = v\right\}$. The extra condition $(5)$ is equivalent to the requirement that $r_v \geq 1$, for every $v\in V(T)$. Trees satisfying it are called \emph{open} trees. In fact, in the constructions of our two particular instances of the free operad, we are going to find only vertices satisfying $r_v \geq 2$. A \emph{reduced} tree is a tree for  which every vertex satisfies this extra condition is called \emph{reduced}.

Let us denote by $\tree (r)$ the category whose objects are $r$-trees and whose morphisms are just isomorphisms. $\widetilde{\tree  (r)}$ denote the full  subcategory of \emph{reduced} trees. For $r\geq 2$, $Y_r \in\widetilde{\tree  (r)} $.

\subsection{The general free operad}\label{free} The forgetful functor $U : \Op \longrightarrow \sMod$ has a left adjoint, the \emph{free operad functor}, $\Gamma : \sMod \longrightarrow \Op$. Arity-wise it can be computed as

$$
\Gamma (M)(l) = \dirlim_{T \in \tree(l)} M(T) \ ,
$$

Here, $M(T)$ denotes the \emph{treewise tensor product} of the $\Sigma$-module $M$ over a tree $T$:

$$
M(T) = \bigotimes_{v\in V(T)} M(r_v) \ .
$$

Of course, this free operad has the well-known universal property of free objects; that is, every morphism of operads $\phi : \Gamma (M) \longrightarrow P$ is uniquely determined by its restriction $\phi_{| M}$. (See \cite{Fre17a}, prop. 1.2.2, for instance.)

\subsection{Two particular instances of the free operad} We need two particular, smaller instances of the free operad.

First, because of \cite{Fre17a}, the restriction of the general free operad $\Gamma$ to $\Sigma$-modules satisfying $M(0) = M(1) = 0$ is a non-unitary and connected operad $\Gamma (M) \in \Op_0$. So the general free operad functor restricts to a smaller one, the \emph{free non-unitary and connected operad}, which we denote $\Gamma_0$. It is the left adjoint of the obvious forgetful functor:

      $$
      \xymatrix{
      {\Op_0}  \ar@<.5ex>[r]^-{U}    &
      \sModu  \ar@<.5ex>[l]^-{\Gamma_0}
      }  \ .
      $$

This is the free operad used by Markl in constructing his minimal models à la Sullivan of non-unitary and cohomologically connected operads (see \cite{Mar96} and \cite{MSS02}).

Second, if $M \in \lModu/\overline{\Com}$, then the general free operad $\Gamma (M)$ inherits the additional structure of an augmented, connected, and unitary $\Lambda$-operad (\cite{Fre17a}, prop. A.3.12). Hence, because of the isomorphism of categories between $\Lambda$-operads and unitary $\Sigma$-operads, it has a unitary extension. Let us denote it by $\Gamma_ {+1}(M) = \Gamma (M)_+$ . It is the left adjoint of the forgetful functor $\overline{U}$ which sends each operad $P$ to its augmentation ideal $\overline{P}$:

      $$
      \xymatrix{
      {\Op_+ = \mathbf{\Lambda}\Op_0 / \Com}  \ar@<.5ex>[r]^-{\overline{U}}    &
      \lModu/\overline{\Com}  \ar@<.5ex>[l]^-{\Gamma_+}
      }  \ .
      $$

A little road map for these categories and functors, where $\iota$ denotes the natural inclusions:

$$
\xymatrix{
{\sMod}\ar@<.5ex>[rr]^{\Gamma}      &         & {\Op}  \ar@<.5ex>[ll]^{U} &        \\
{\sModu} \ar[u]^{\iota} \ar@<.5ex>[r]^-{\Gamma_0}    &  {\Op_0} \ar[ur]^{\iota}  \ar@<.5ex>[l]^-{U} &     & {\Op_+} \ar[ul]^{\iota} \ar[ll]^{\tau}  \ar@{=}[r]  &  {\mathbf{\Lambda}\Op_0/\Com} \ar@<.5ex>[r]^{\overline{U}}  & {\lModu/\overline{\Com} \ .}  \ar@<.5ex>[l]^{\Gamma_+}
}
$$

\begin{remark} Let us point out how $\Gamma_+(M)$ inherits its $\Lambda$-structure from the one of $M \in \lModu/\overline{\Com}$. The $\Lambda$-structure of the latter consists of the restriction operations and the augmentations

$$
\delta_i : M(n) \longrightarrow M(n-1) \ , i = 1, \dots , n \ ,
\qquad \text{and} \qquad
\varepsilon : M(n) \longrightarrow \mk
$$

for $n \geq 2$ ($M(0) = M(1) = 0$). The corresponding morphisms on $\Gamma_+(M)$ are obtained as follows (see \cite{Fre17a}, page 83):

\begin{itemize}
  \item[$\bullet$] Restriction operations $\delta_i : \Gamma_+(M)(n) \longrightarrow \Gamma_+(M)(n-1)$ are determined on $M(n)$ by:
      \begin{itemize}
        \item[(1)] the augmentation $\varepsilon : M(2) \longrightarrow \mk = \Gamma_+(M)(1)$,
        \item[(2)] the internal restriction morphisms $\delta_i : M(n) \longrightarrow M(n-1)$, for $n\geq 3$.
      \end{itemize}
  \item[$\bullet$] The augmentation $\varepsilon : \Gamma_+(M)(n) \longrightarrow \mk$ is determined by $\varepsilon : M(n) \longrightarrow \mk$.
\end{itemize}

Notice that, as consequence $\delta_0 = \delta_1 = \varepsilon : \Gamma_+(M)(2) \longrightarrow \Gamma_+(M)(1) = \mk$. This means that free unitary operads $\Gamma_+(M)$ are naturally operads with a unitary multiplication whenever we define $\varepsilon : M(2) \longrightarrow \mk$ to be the generator of $\mk$ on one element $\omega \in M(2)$. Then, for $\omega \in \Gamma_+(M)$, we have $\omega \circ_1 1 = \id = \omega \circ_2 1$.
\end{remark}

The key point that encompasses the possibility of constructing minimal models for operads in both cases we are studying, cohomologically non-unitary and unitary, is that, since $M(0) = M(1) = 0$, there are no arity zero and one trees in the colimit defining the free operad. In this case, the only morphisms in the subcategory of open trees $\widetilde{\tree(l)} $ are trivial isomorphisms. As a consequence, this colimit is reduced to a direct sum \cite{Fre17a}, proposition A.3.14. Hence, for $\Gamma = \Gamma_0, \Gamma_+$,

      $$
      \Gamma (M)(l) = \bigoplus_{T \in \widetilde{\tree(l)}} M(T) \ .
      $$

All this leads to the following

\begin{lemma}\label{lamadredelcordero} For every module $M$ with $M(0) = M(1) = 0$, and every homogeneous module $E$ of arity $p > 1$, $\Gamma_0$ and $\Gamma_+$ verify:
\begin{itemize}
  \item[(a)]

$$
\Gamma_0(M)(l) =
\begin{cases}
   0 , & \mbox{if } l= 0, \\
   \mk , & \mbox{if } l = 1,  \\
   \Gamma (M) (l) , & \mbox{if } l \neq 0,1
\end{cases}
\qquad \text{and} \qquad
\Gamma_+(M) (l) =
\begin{cases}
   \mk, & \mbox{if } l= 0, \\
   \mk , & \mbox{if } l = 1, \\
   \Gamma (M) (l) , &\mbox{if } l \neq 0,1 \ .
\end{cases}
$$
  \item[(b)] For $\Gamma = \Gamma_0, \Gamma_+$,
$$
\Gamma (M \oplus E) (l) =
\begin{cases}
  \Gamma (M) (l), & \mbox{if } l < p, \\
  \Gamma (M) (p) \oplus E , & \mbox{if } l = p.
\end{cases}
$$
\end{itemize}
\end{lemma}

\begin{proof} Let's compute:

$$
  (M\oplus E)(T) =  \bigotimes_{v\in V(T)} (M(r_v)\oplus E(r_v) )
    = \left(  \bigotimes_{ \substack{ v\in V(T)  \\   r_v \neq p}  } M(r_v) \right) \otimes \left( \bigotimes_{ \substack{ v\in V(T) \\ r_v = p  }  } \left(  M(r_v) \oplus E \right) \right) \ .
$$

Thus,

$$
  \Gamma (M\oplus E)(m)
    = \bigoplus_{T \in \widetilde{\tree (m)}}  \left[ \left( \bigotimes_{ \substack{ v\in V(T)  \\   r_v \neq p}  } M(r_v) \right) \otimes \left(  \bigotimes_{ \substack{ v\in V(T) \\ r_v = p  }  } \left(  M(r_v) \oplus E \right)   \right)  \right] \ .
$$

If $m<p$,  $ \bigotimes_{ \substack{ v\in V(T) \\ r_v = p  }  } \left(  M(r_v) \otimes E \right) = \mk $, since there are no trees with $m < p$ ingoing edges and a vertex with $p$ ingoing edges (there are no corks, since $M(0) = 0$). Hence, in this case, we simply have

$$
\Gamma (M\oplus E)(m) = \bigoplus_{T \in \widetilde{\tree (m)}}   \left( \bigotimes_{ \substack{ v\in V(T)  \\   r_v \neq p}  } M(r_v) \right) = \Gamma (M) (m) \ .
$$

For $ m = p$, we split our sum over all trees in two terms: one for the corolla $Y_p$ and another one for the rest:

\begin{multline*}
\Gamma (M\oplus E)(p) =
\left[ \left( \bigotimes_{ \substack{ v\in V(Y_p)  \\   r_v \neq p}  } M(r_v) \right) \otimes \left(  \bigotimes_{ \substack{ v\in V(Y_p) \\ r_v = p  }  } \left(  M(r_v) \oplus E \right)   \right)  \right]
\oplus  \\
\bigoplus_{ \substack{  T \in \widetilde{\tree (m)} \\ T \neq Y_p  } }  \left[ \left( \bigotimes_{ \substack{ v\in V(T)  \\   r_v \neq p}  } M(r_v) \right) \otimes \left(  \bigotimes_{ \substack{ v\in V(T) \\ r_v = p  }  } \left(  M(r_v) \oplus E \right)   \right)  \right]
\end{multline*}

And we have:

\begin{itemize}
  \item[(1)] The set of vertices $ v\in V(Y_p) \ ,   r_v \neq p  $ is empty. So for the first tensor product we get: $ \bigotimes_{ \substack{ v\in V(Y_p)  \\   r_v \neq p}  } M(r_v) = \mk  $.
  \item[(2)] There is just one vertex in $v\in V(Y_p)$ with $ r_v = p$. Hence, $\bigotimes_{ \substack{ v\in V(Y_p) \\ r_v = p  }  }   (M(r_v) \oplus E) = M(p)\oplus E $.
  \item[(3)] We leave  $\bigotimes_{ \substack{ v\in V(T)  \\   r_v \neq p}  } M(r_v)  $ as it is.
  \item[(4)] As for  $ \bigotimes_{ \substack{ v\in V(T) \\ r_v = p  }  } \left(  M(r_v) \oplus E \right) $, since we are also assuming $r_v \geq 2$, this set of vertices is empty for every tree. So we only get $\mk$ for every vertex $v$.
\end{itemize}

All in all,

$$
\Gamma (M\oplus E)(p) = \left(  M(p) \oplus E \right)\ \oplus   \bigoplus_{ \substack{  T \in \widetilde{\tree (m)} \\ T \neq Y_p  } }  \left[  \bigotimes_{ \substack{ v\in V(T)  \\   r_v \neq p}  } M(r_v)  \right] = \Gamma (M)(p)  \oplus E \ .
$$

\end{proof}

\begin{remark} So, for $M(0) = M(1) = 0$, and forgetting the $\Lambda$-structure when there is one, it's clear that both $\Gamma_0(M)$ and $\Gamma_+(M)$ agree with the general free operad $\Gamma (M)$, outside arities $0$ and $1$. By definition, also $\Gamma_+(M) = \Gamma_0 (M)_+$, when $M$ has a $\Lambda$-module structure.
\end{remark}

\section{Basic operad homotopy theory}

The Hinich \cite{Hin97} model structure on the category of operads can be restricted to the subcategories of non-unitary and connected $\Op_0$, unitary and connected $\Op_+$ and then to their comma categories such as $\Mag_* \backslash \Op$. However, in order to talk about minimal models, we do not need the full power of the whole model structure, but just three elements: weak equivalences, or \emph{quis}, homotopy, and cofibrant (minimal) objects. This basic homotopy theory can be formalized under the name of \emph{Cartan-Eilenberg}, or \emph{Sullivan} categories (see \cite{GNPR10}). We develop here this basic homotopy theory for operads.

Cofibrant, \emph{Sullivan}, operads are built by the known procedure of \lq\lq attaching cells". We have to distinguish between the non-unitary case and the unitary one from the very definition stage of this attachment procedure because, in the second case, we also need to deal with the restriction operations; that is, the action of the units. For this, our strategy is the following. First, we treat the non-unitary case. Then, we see the necessary changes we have to perform on it, once we add the data of the restriction operations.

\subsection{Lifting properties: the non-unitary case}

For the sake of lightening the notation, in this section $\Gamma$ refers to its restriction to non-unitary, connected operads, $\Gamma_0$.

\begin{definition}\label{OpdefKS} (See \cite{MSS02}, cf \cite{GNPR05}) Let $n\geq 2$ be an integer. Let $P \in \Op_0$ be free as a graded operad, $P= \Gamma (M)$, where $M $ is a graded $\Sigma$-module, with $M(0)=M(1)=0$. An \textit{arity $n$ principal extension} of $P$ is the free graded operad

$$
P \sqcup_d \Gamma (E) :=\Gamma (M \oplus E) \ ,
$$

where $E$ is an arity-homogeneous $\Sigma_n$-module with zero differential and $d:E\longrightarrow ZP(n)^{+1}$  a map of $\Sigma_n$-modules of degree $+1$. The differential $\partial$ on $P\sqcup_d \Gamma (E)$ is built upon the differential of P, $d$, and the Leibniz rule.
\end{definition}

\begin{remark} In the context of commutative dg algebras, the analogous construction is called a \emph{Hirsch extension} \cite{GM13}, or a \emph{KS-extension} \cite{Hal83}.
\end{remark}

\begin{lemma}\label{diferencialHirsch} $P \sqcup_d \Gamma (E) $ is a dg operad, and the natural inclusion $\iota : P \longrightarrow P \sqcup_d \Gamma (E) $ is a morphism of dg operads.
\end{lemma}

\begin{proof} This is clear.
\end{proof}

\begin{lemma}[Universal property of principal extensions]\label{propuniversalHirsch}
Let $P \sqcup_d \Gamma (E)$ be a principal extension of a free-graded operad $P = \Gamma (M)$, and let
$\varphi : P \to Q$ be a morphism of operads. A morphism
$\psi : P\sqcup_d \Gamma (E)\lra Q$ extending $\varphi$ is uniquely determined by a morphism of $\Sigma_n$-modules
$f : E\to Q(n)$  satisfying $\partial f = \varphi d$.
\end{lemma}

\begin{proof} This is clear.
\end{proof}

\begin{lemma}\label{liftingthroughextensions} Let $\iota: P \longrightarrow P \sqcup_d \Gamma (E)$ be an arity $n$ principal-extension and
$$
\xymatrix{
{P}   \ar[r]^{\varphi}  \ar[d]_{\iota}     &     Q \ar@{>>}[d]^{\rho}_{\wr} \\
{P\sqcup_d \Gamma(E)}  \ar[r]^{\psi}\ar@{-->}[ur]^{\psi'}  &     R
}
$$
a solid commutative diagram of operad morphisms, where $\rho$ is a surjective quasi-isomorphism.
Then, there is an operad morphism $\psi'$ making both triangles commute.
\end{lemma}

\begin{proof} Thanks to the previous lemma \ref{propuniversalHirsch}, we can build $\psi'$ from a $\mk[\Sigma_n]$-linear map $f: E\longrightarrow Q(n)$, verifying $\rho f = \psi_{| E}$ and $\varphi d =\partial f$. The reader interested in the details can take the proof of the analogous result in \cite{CR19}, lemma 3.4, and make the obvious changes.
\end{proof}

\subsection{Lifting properties: the unitary case}

For the sake of lightening the notation, in this section $\Gamma$ refers to its restriction to unitary, connected operads, $\Gamma_+$.

\begin{definition}\label{OpdefKS+} Let $n\geq 2$ be an integer. Let $P \in \Op_+$ be free as a graded operad, $P= \Gamma (M)$, where $M $ is a graded $\Sigma$-module, with $M(0)=M(1)= 0$. A \textit{unitary arity $n$ principal extension} of $P$ is the free graded operad

$$
P \sqcup_d^\delta \Gamma (E) :=\Gamma (M \oplus E) \ ,
$$

where $E$ is an arity-homogeneous $\Sigma_n$-module with zero differential and:

\begin{enumerate}
  \item[(a)] $d:E\longrightarrow ZP(n)^{+1}$  is a map of $\Sigma_n$-modules of degree $+1$. The differential $\partial$ on $P\sqcup_d \Gamma (E)$ is built upon the differential of P, $d$ and the Leibniz rule.
  \item[(b)] $\delta_i : E \longrightarrow M(n-1), i = 1, \dots , n$ are morphisms of $\mk$-graded vector spaces, compatible with $d$ and the differential of $P$, in the sense that, for all $i=1, \dots , n$ we have commutative diagrams

      $$
      \xymatrix{
      {E} \ar[r]^{d} \ar[d]^{\delta_i}   &  {ZP(n)^{+1}} \ar[d]^{\delta_i}  \\
      {M(n-1)}   \ar[r]^{\partial}       &   {P(n-1)^{+1}} \ .
      }
      $$

      \noindent They have also to be compatible with the $\Lambda$-structure of $P$, from arity $n-1$ downwards.
\end{enumerate}
\end{definition}

\begin{lemma}\label{diferencialHirsch+} $P \sqcup_d^\delta \Gamma (E) $ is a unitary dg operad and the natural inclusion $\iota : P \longrightarrow P \sqcup_d^\delta \Gamma (E) $ is a morphism of unitary dg operads.
\end{lemma}

\begin{proof} This is clear.
\end{proof}

\begin{lemma}[Universal property of unitary principal extensions]\label{propuniversalHirsch+}
Let $P \sqcup_d^\delta \Gamma (E)$ be a unitary principal extension of a free-graded unitary operad $P = \Gamma (M)$, and let
$\varphi : P \to Q$ be a morphism of unitary operads. A morphism
$\psi : P\sqcup_d^\delta \Gamma (E)\lra Q$ extending $\varphi$ is uniquely determined by a morphism of $\Sigma_n$-modules $f : E\to Q(n)$  satisfying $\partial f = \varphi d$ and making commutative the diagrams

      $$
      \xymatrix{
      {E} \ar[r]^{f} \ar[d]^{\delta_i}   &  {Q(n)} \ar[d]^{\delta_i}  \\
      {M(n-1)}   \ar[r]^{\varphi}       &   {Q(n-1)}
      }
      $$

for all $i=1, \dots , n$.
\end{lemma}

\begin{proof} This is clear.
\end{proof}

\begin{lemma}\label{liftingthroughextensions+} Let $\iota: P \longrightarrow P \sqcup_d^\delta \Gamma (E)$ be a unitary arity $n$ principal-extension and
$$
\xymatrix{
{P}   \ar[r]^{\varphi}  \ar[d]_{\iota}     &     Q \ar@{>>}[d]^{\rho}_{\wr} \\
{P\sqcup_d^\delta \Gamma(E)}  \ar[r]^{\psi}\ar@{-->}[ur]^{\psi'}  &     R
}
$$
a solid commutative diagram of morphisms of operads, where $\rho$ is a surjective quasi-isomorphism between operads with unitary multiplication. Then, there is an operad morphism $\psi'$ making both triangles commute.
\end{lemma}

\begin{proof} We have built $\psi'$, lifting $\psi$ and extending $\varphi$ in the non-unitary case, thanks to $f: E \longrightarrow Q(n)$, verifying $\rho f = \psi_{| E}$ and $\partial f = \varphi d$. Now, we need to check the compatibility of $f$ with the restriction operations. That is, we would like to have $\delta_i (fe) = \varphi (\delta_i e) \ , \quad \text{for all} \ i = 1, \dots , n$. This is not necessarily true, so for every $i=1, \dots, n$, we consider the differences $\omega_i = \delta_i fe - \varphi \delta_i e$ and check the following:

\begin{itemize}
  \item[(i)] For every $i$, $\omega_i$ is a cocycle:  $\partial \omega_i = \partial \delta_ife - \partial \varphi\delta_i e = \delta_i\partial fe -\delta_i\partial fe = 0$.
  \item[(ii)] For every $i$, $\omega_i $ belongs to $\ker\rho$: $\rho \omega_i = \rho \delta_i fe - \rho\varphi\delta_i e = \delta_i \rho fe -\psi\iota\delta_i e = \delta_i\psi e - \psi\delta_i e = \psi\delta_i e - \psi\delta_i e = 0$.
  \item[(iii)] The family $\left\{ \omega_i \right\}_{i=1,\dots n}$ satisfies the Kan-like condition, \ref{Kan-likecondition}: $\delta_i\omega_j = \delta_i\delta_j fe -\delta_i\varphi\delta_j e = \delta_i\delta_jfe -\delta_i\delta_j\varphi e = \delta_{j-1}\delta_i fe - \delta_{j-1}\delta_i \varphi e = \delta_{j-1} \omega_i$, for $i<j$
\end{itemize}

Here, we have used the fact that $\varphi, \rho, \psi$ are already unitary $= \Lambda$-morphisms. Hence, because of lemma \ref{enhancedKan-likeresult}, we conclude that there is $\omega \in ZQ(n) \cap \ker\rho$ such that $\delta_i \omega = \omega_i $ for all $i=1,\dots , n$. So, we change our original $fe$ into $f'e = fe - \omega$. We immediately see that we haven't lost what we had obtained in the non-unitary case, namely $\rho f'e = \psi e$ and $\partial f' e = \varphi de$ and we have added the commutativity with the restriction operations: $
\delta_i f'e = \delta_i fe - \delta_i\omega = \delta_i fe - \omega_i = \delta_ife -(\delta_i fe - \varphi\delta_i e)  = \varphi\delta_i e$.

Final problem: this $\omega$ should depend $\mk[\Sigma_n]$-linearly on $e \in E(n)$ in order to have a map $f' : E(n) \longrightarrow Q(n)$: $e \mapsto \{ \omega_1(e), \dots , \omega_n(e)\} \mapsto \omega(e)$. $\mk$-linearity is clear at both steps, just looking at the definitions of $\omega_i$ and the algorithm producing $\omega$ in the proof of lemma \ref{Kan-likeresult}. However, $\Sigma_n$-equivariance is unclear. First, $\delta_i$'s appearing in the definition of $\omega_i$'s, are \emph{not} $\Sigma_n$-equivariant. Neither $s_i$'s used in the construction of $\omega$ from $\omega_i$'s are. Instead, for $\sigma \in \Sigma_n$ and $\nu$ an arbitrary element, we have the following relations, which are easy to verify: $\delta_i (\sigma\cdot  \nu) = \delta_{\sigma^{-1}(i)} (\nu)$ and $s_i (\sigma \cdot \nu) = s_{\sigma^{-1}(i)} (\nu)$. These relations pass to $\omega_i$: $\omega_i(\sigma\cdot e) = \omega_{\sigma^{-1}(i)} (e)$.

Fortunately, we have the standard \emph{average} procedure to get a $\Sigma_n$-equivariant morphism from $\omega$: $\widetilde{\omega} (e) = \frac{1}{n!} \sum_{\sigma \in \Sigma_n} \sigma \cdot \omega (\sigma^{-1} \cdot e)$. $\widetilde{\omega} (e)$ is still a cocycle and belongs to $\ker \rho$. And we still have what we were looking for from the very beginning:

\begin{eqnarray*}
  \delta_i\widetilde{\omega} (e) &=& \delta_i\left( \frac{1}{n!} \sum_{\sigma \in \Sigma_n} \sigma \cdot \omega (\sigma^{-1} \cdot e) \right) = \frac{1}{n!} \sum_{\sigma \in \Sigma_n} \delta_i \left(\sigma \cdot \omega (\sigma^{-1} \cdot e) \right)\\
   &=& \frac{1}{n!} \sum_{\sigma \in \Sigma_n} \delta_{\sigma^{-1}(i)} \left( \omega (\sigma^{-1} \cdot e) \right) = \frac{1}{n!} \sum_{\sigma \in \Sigma_n} \omega_{\sigma^{-1}(i)} (\sigma^{-1}\cdot e) = \frac{1}{n!} \sum_{\sigma \in \Sigma_n} \omega_i(e) = \omega_i(e) \ .
\end{eqnarray*}

Therefore, we put this $\widetilde{\omega}$ instead of our first $\omega$, and we are done.
\end{proof}

\subsection{Lifting properties: conclusions}

\begin{definition} A \textit{Sullivan operad} is the colimit of a sequence of principal extensions of arities $l_n \geq 2$, starting from the initial operad.

$$
 I_\alpha \longrightarrow P_1 = \Gamma (E(l_1))  \longrightarrow  \cdots \longrightarrow  P_{n} = P_{n-1} \sqcup \Gamma (E(l_n)) \longrightarrow \dots \longrightarrow \dirlim_n P_n = P_\infty \ .
$$
\end{definition}

This definition stands for both the non-unitary and unitary cases, as long as we read it with the following dictionary, where the first option is for the non-unitary case, and the second for the unitary one: \label{diccionari} (1) $\alpha = 0, +$, (2) $P = P, P_+$, (3) $\Gamma = \Gamma_0, \Gamma_+$, (4) $\sqcup = \sqcup_d, \sqcup_d^\delta$.

The next result says that Sullivan operads are cofibrant objects in the Hinich  model structure of the category of operads, \cite{Hin97}.

\begin{proposition}\label{strict_lifting}
Let $S$ be a Sullivan operad. For every solid diagram of operads (resp., of operads with unitary multiplication)

$$
\xymatrix{
&P\ar@{->>}[d]^{\rho}_{\wr}\\
\ar@{.>}^{\varphi'}[ur]S\ar[r]^\varphi&Q&
}
$$

in which $\rho$ is a surjective quasi-isomorphism, there exists a morphism of operads (resp., of operads with unitary multiplication) $\varphi'$ making the diagram commute.
\end{proposition}

\begin{proof} Induction and lemmas \ref{liftingthroughextensions}, or \ref{liftingthroughextensions+}, depending on wether we are dealing with the non-unitary, or the unitary case.
\end{proof}

\subsection{Homotopy} Similarly to the setting of commutative algebras, there is a notion of homotopy between morphisms of operads, defined via a functorial path (see Section 3.10 of \cite{MSS02}, cf. \cite{CR19}), based on the following remark.

\begin{remark} Let $P$ be a dg operad and $K$ a commutative dg algebra. Then $P \otimes K = \left\{ P(n)\otimes K\right\}_{n \geq 0}$ has a natural operad structure given by the partial composition products $(\omega \otimes a) \circ_i (\eta \otimes b) = (-1)^{|a||\eta|} (\omega\circ_i\eta)\otimes (ab)$.
\end{remark}

In particular, let $K = \mk [t,dt] = \Lambda (t,dt)$ be the free commutative dg algebra on two generators $|t| = 0$, $|dt| = 1$ and differential sending $t$ to $dt$. We have the unit $\iota $ and evaluations $\delta^0$ and $\delta^1$ at $t=0$ and $t=1$ respectively,
which are morphisms of $\Com$-algebras satisfying $\delta^0 \circ \iota=\delta^1 \circ \iota=\id$.

\[
\xymatrix{\mk\ar[r]^-{\iota}&\mk[t,dt] \ar@<.6ex>[r]^-{\delta^1} \ar@<-.6ex>[r]_-{\delta^0}&\mk}\,\,\,;\,\,\, \delta^k\circ \iota=\id  \ .
\]

The following are standard consequences of Proposition $\ref{strict_lifting}$. Proofs are easy adaptations of the analogous results in the setting of $\Com$-algebras (see Section 11.3 of \cite{GM13}; see also \cite{CR19} in the context of operad algebras).

\begin{remark} We can skip treating the non-unitary and unitary cases separately in this section, once we notice that:
\begin{itemize}
  \item[(a)] If $P_+$ is a (cohomologically) unitary operad, then so is $P_+[t,dt]$, thanks to $1 \in \mk [t,dt]$.
  \item[(b)] Whenever we state the existence of a morphism, we either use a universal property valid in both cases, or we are using the appropriate version of proposition \ref{strict_lifting} above.
\end{itemize}
\end{remark}

\begin{definition}\label{path} A \textit{functorial path} in the category of operads is defined as the functor $-[t,dt]:\Op \longrightarrow \Op $, given on objects by $P[t,dt]=P\otimes \mk[t,dt]$ and on morphisms by $\varphi [t,dt]=\varphi\otimes\mk[t,dt]$, together with the natural transformations

$$
\xymatrix{P\ar[r]^-{\iota}&P[t,dt] \ar@<.6ex>[r]^-{\delta^1} \ar@<-.6ex>[r]_-{\delta^0}&P}\,\,\,;\,\,\, \delta^k\circ \iota= \id \ ,
$$

where

$$
\delta^k=1\otimes \delta^k:P[t,dt]\longrightarrow P\otimes\mk=P
\quad \text{and} \iota=1\otimes\iota:P=P\otimes\mk\to P[t,dt].
$$
\end{definition}

The map $\iota$ is a quasi-isomorphism of operads while the maps $\delta^0$ and $\delta^1$ are
surjective \emph{quis} of operads. The functorial path gives a natural notion of homotopy between morphisms of operads:

\begin{definition}
Let $\varphi, \psi :P \longrightarrow Q$ be two morphisms of operads. A \textit{homotopy from $\varphi$ to $\psi$} is given by a morphism of operads $H:P\longrightarrow Q[t,dt]$ such that $\delta^0\circ H=\varphi$ and $\delta^1\circ H=\psi$. We use the notation $H:f\simeq g$.
\end{definition}

The homotopy relation defined by a functorial path is reflexive and compatible with the composition. Furthermore, the symmetry of $\Com$-algebras $\mk[t,dt]\lra \mk[t,dt]$ given by $t\mapsto 1-t$ makes the homotopy relation into a symmetric relation. However, the homotopy relation is not transitive in general. As in the rational homotopy setting of $\Com$-algebras, we have:

\begin{proposition}\label{transitive}
The homotopy relation between morphisms of operads is an equivalence relation for those morphisms whose source is a Sullivan operad.
\end{proposition}

\begin{proof} The interested reader can take the proof of the analogous result in \cite{CR19}, proposition 3.10, and make the obvious changes.
\end{proof}

Denote by $[S,P]$ the set of homotopy classes of morphisms of operads $\varphi : S\longrightarrow P$.

\begin{proposition}\label{bijecciohomotopies}Let $S$ be a Sullivan operad. Any quasi-isomorphism
$\varpi: P\longrightarrow Q$ of operads induces a bijection $\varpi_*:[S,P] \longrightarrow [S,Q]$.
\end{proposition}

\begin{proof} See the proof of the analogous result in \cite{CR19}, proposition 3.11.
\end{proof}

\section{Minimal models}

Sullivan \emph{minimal} operads are Sullivan operads for which the process of adding new generators $E$ is done with strictly increasing arities. In this section, we prove the existence and uniqueness of Sullivan's minimal models for operads in our two cases mentioned above: (cohomologically) non-unitary and unitary.

\begin{definition}\label{defSullmin} A \textit{Sullivan minimal operad} $P_\infty$ is the colimit of a sequence of principal extensions starting from the initial operad, ordered by strictly increasing arities

$$
I_\alpha \longrightarrow P_2 = \Gamma (E(2))  \longrightarrow  \cdots \longrightarrow  P_{n} = P_{n-1} \sqcup \Gamma (E(n)) \longrightarrow \dots \longrightarrow \dirlim_n P_n = P_\infty \ ,
$$

with $E(n)$ an arity $n$ homogeneous $\Sigma_n$-module with zero differential. A \emph{Sullivan minimal model} for an operad $P$ is a Sullivan minimal operad $P_\infty$ together with a quasi-isomorphism $\rho: P_\infty \stackrel{\sim}{\longrightarrow} P$.
\end{definition}

\begin{remarks}
\begin{itemize}
\item[(1)] The same conventions as in \ref{diccionari} apply here, so this definition works for both cases: non-unitary and unitary one.
\item[(2)] In particular, a Sullivan minimal operad is a free graded operad $P_\infty = \Gamma (E)$, with $E = \bigoplus_n E(n)$, plus an extra condition on its differential $\partial$, usually called being \emph{decomposable}. The interested reader can check that both definitions, as a colimit of principal extensions or as a free graded operad plus a decomposable differential, agree by looking at \cite{GNPR05}, proposition 4.4.1. Even though this second characterization is useful in practice to recognize a Sullivan minimal operad, we are not going to use it in this paper.
\end{itemize}
\end{remarks}

\subsection{Existence. The non-unitary case}

\begin{theorem}\label{existencia} Every cohomologically non-unitary $HP(0) = 0$ and cohomologically connected $HP(1)\allowbreak = \mk$ operad $P$ has a Sullivan minimal model $P_\infty \stackrel{\sim}{\longrightarrow} P$. This minimal model is non-unitary $P_\infty(0) = 0$ and connected $P_\infty(1) = \mk$.
\end{theorem}

\begin{proof} Let $P$ be a cohomologically connected, cohomologically non-unitary operad. This is Markl's case \cite{MSS02}, with the slight improvement that we are just assuming $HP(0) =0$ instead of $P(0) = 0$. We are going first to write down the proof for this case, and then comment on the modifications needed for the cohomologically unitary case.

Here, we use the free operad functor $\Gamma = \Gamma_0$ and start with $E = E(2) = HP(2)$. Take a $\mk[\Sigma_2]$-linear section $s_2 : HP(2) \longrightarrow ZP(2) \subset P(2)$ of the projection $\pi_2 : ZP(2) \longrightarrow HP(2)$, which exists because $\mk$ is a characteristic zero field, and define:

$$
P_2 = \Gamma (E) \ , \quad {\partial_2}_{|E} = 0 \ , \qquad \text{and}\qquad \rho_2 : P_2 \longrightarrow P \ , \quad {\rho_2}_{|E} = s_2 \ .
$$

It's clear that $P_2$ is a dg operad with differential $\partial_2 = 0$ and $\rho_2$ a morphism of dg operads. Also it is a \emph{quis} in arities $\leq 2$ because:

\begin{itemize}
  \item[(0)] $P_2(0) = \Gamma (E)(0) = 0 = HP(0)$, because of lemma \ref{lamadredelcordero} (a),
  \item[(1)] $P_2(1) = \Gamma (E)(1) = \mk = HP(1) $, because of lemma \ref{lamadredelcordero} (a), and
  \item[(2)] $P_2(2) = \Gamma (E)(2) = E(2) = HP(2)$, because of lemma \ref{lamadredelcordero} (b).
\end{itemize}

Assume we have constructed a morphism of dg operads $\rho_{n-1} : P_{n-1} \longrightarrow P$  in such a way that:

\begin{enumerate}
\item $P_{n-1}$ is a minimal operad, and
\item $\rho_{n-1} : P_{n-1} \longrightarrow P$ is a {\it quis\/} in arities $\leq n-1$.
\end{enumerate}

To build the next step, consider the $\Sigma_n$-module of the relative cohomology of $\rho_{n-1} (n): P_{n-1}(n) \longrightarrow P(n)$, $E = E(n) = H(P_{n-1}(n),P(n))$. Since we work in characteristic zero, every $\Sigma_n$-module is projective. So we have a $\Sigma_n$-equivariant section $s_n = (d_n \ f_n)$ of the projection $P_{n-1}(n)^{+1} \oplus P(n) \supset Z(P_{n-1}(n),P(n))\longrightarrow H(P_{n-1}(n),P(n))$. Let

$$
\begin{pmatrix}
  -\partial_{n-1} (n)          & 0 \\
  -\rho_{n-1}(n) & \partial (n)
\end{pmatrix}
$$

be the differential of the mapping cone $C\rho_{n-1}(n)$: the cocycle condition implies that $
\partial_{n-1}(n)d_n = 0$ and $ \rho_{n-1}d_n = \partial_n(n)f_n$. That is, $d_n$ induces a differential $\partial_n$ on $P_n = P_{n-1} \sqcup_{d_n} \Gamma (E)$ and $f_n$ a morphism of operads $\rho_n : P_n \longrightarrow P$ such that ${\rho_n}_{|P_{n-1}} = \rho_{n-1}$ and ${\rho_n}_{|E'} = f_n$, because of lemmas \ref{diferencialHirsch} and \ref{propuniversalHirsch}.

The verification that $\rho_n$ induces an isomorphism in cohomology ${\rho_n}_* : H P_n(m) \longrightarrow HP(m)$ in arities $m= 0,\dots , n$ relies on lemma \ref{lamadredelcordero}. First, if $m < n$, $\rho_n (m) = \rho_{n-1} (m)$ by lemma \ref{lamadredelcordero}  and so, by the induction hypothesis, we are done. Again by lemma \ref{lamadredelcordero} in arity $n$, we have:

$$
\rho_n (n) = (\rho_{n-1}(n) \quad f_n ) : P_{n-1}(n) \oplus E(n) \longrightarrow P(n) \ .
$$

Checking that this $\rho_n (n)$ is a \emph{quis} is now an easy exercise.
\end{proof}

\subsection{Existence. The unitary case}

\begin{theorem}\label{existencia+} Every cohomologically unitary $HP(0) = \mk$, cohomologically connected $HP(1) =\mk$ operad with unitary multiplication $P\in \Mag_+ \backslash \Op $, has a Sullivan minimal model $P_\infty \stackrel{\sim}{\longrightarrow} P$. This minimal model is unitary $P_\infty(0) = \mk$, connected $P_\infty(1) = \mk$ and has a unitary multiplication.
\end{theorem}

\begin{proof} Let $P$ be a cohomologically connected and cohomologically unitary operad. We make the most of the non-unitary construction we have already built, but now using $\Gamma = \Gamma_+$. To that process we have to add two things:

\begin{enumerate}
  \item[(1)] A $\Lambda$-structure on the new generators $E$, and
  \item[(2)] The need to check the compatibility of the differentials $d_n$ and the successive extensions of our \emph{quis} $f_n$ with this $\Lambda$-structures.
\end{enumerate}

So, starting with $E = E(2) = HP(2)$, we have natural choices for the restriction operations: the ones induced in cohomology by those of $HP(2)$. Next, we have to find a section $s'_2$ that makes the following diagram commute for $i=1,2$:

$$
\xymatrix{
{E =HP(2)} \ar[d]^{\delta_i} \ar@{.>}@/^1pc/[r]^{s'_2}   &  {ZP(2)}  \ar[d]^{\delta_i} \ar[l]^{\pi_2}     \\
{I_+(1) = \mk = HP(1)}  \ar@/^1pc/[r]^{s_1}    &  {ZP(1)} \ar[l]^-{\pi_1}
}
$$

Here, the section $s_1$ is the unique morphism from the initial operad $I_+$ and the $\Lambda$-structure on $E$ is the one induced by $\delta_i : P(2) \longrightarrow P(1)$ on cohomology. Notice that $s_1$ is necessarily a section of $\pi_1$. So we are seeking a section $s'_2$ that satisfies $\delta_i s'_2 = s_1 \delta_i$ for $ i=1,2$, so the induced morphism $\rho_2 : P_2 = \Gamma (E) \longrightarrow P$ to be unitary.

We work as before: given $e\in E$, and any section $s_2$ obtained from the non-unitary case, we check that the elements $\omega_i = \delta_i s_2 e - s_1\delta_i e, i = 1,2$ are cocycles satisfying the Kan condition \ref{Kan-likecondition}. Moreover: they are coboundaries: $  \pi_1\omega_i = \pi_1 (\delta_is_2 e - s_1\delta_i e) =  \pi_1\delta_i s_2 e - \pi_1 s_1 \delta_i e   = \delta_i\pi_2s_2 e - \pi_1 s_1 \delta_i e =  \delta_i e - \delta_i e = 0$ for $i = 1,2$. Here we have used the fact that the projections onto the cohomology $\pi_1, \pi_2$ are $\Lambda$-morphisms and $s_1, s_2$ are sections of them. Hence, \ref{enhancedKan-likeresult} tells us that there is a coboundary $\partial\omega \in BP(2)$ such that $\delta_i \partial\omega = \omega_i, i=1,2$. So, we subtract this $\partial\omega$ from our previous, arbitrary section $s_2$: $s'_2 e = s_2 e - \partial\omega$, perform the average procedure, and check we still have a \emph{quis}.

Assume we have already extended the $\Lambda$-structure with a compatible \emph{quis} up to arity $n-1, n>2$. For the next step, we have to produce a $\Lambda$-structure on $E = E(n) = HC\rho_{n-1}(n) = H(P_{n-1}(n), P(n))$ and a compatible section $s'_n = (d'_n \ f'_n)^t$; that is, making the following diagram commute for $i=1,\dots , n$:

$$
\xymatrix{
{E} \ar[d]^{\delta_i} \ar@{.>}@/^1pc/[r]^-{s'_n}    & {Z(P_{n-1}(n), P(n))}  \ar[d]^{\delta_i} \ar[l]^-{\pi_n}     \\
M_{n-1}(n-1) \ar[r]^-{s_{n-1}}     &  {Z(P_{n-1}(n-1), P(n-1))}
}
$$

Here $s_{n-1} = (\partial (n-1) \ \rho(n-1))^t$. So again we need our section $s'_n$ to verify $
\delta_i s'_n = s_{n-1}\delta_i, \ , \quad \text{for} \ i=1,\dots , n$, since then the induced morphism $\rho_n : P_n = P_{n-1} \sqcup_d^\delta \Gamma (E) \longrightarrow Q$ will be a unitary operad morphism. An always available choice for the $\Lambda$-structure on $E$ is $\delta_i = 0$ for all $i=1, \dots , n$ (see remark \ref{unitsremainforever} below). Hence we have to produce a section $s'_n$ such that $\delta_i s'_n = 0, i = 1, \dots , n$: take the section $s_n$ we got from the non-unitary case, $e \in E$ and compute: $\pi_{n-1}\delta_i s_n e = 0$. Here, $\pi_{n-1} : Z(P_{n-1}(n-1), P(n-1)) \longrightarrow H(P_{n-1}(n-1), P(n-1)) = 0$, because $\rho(n-1)$ is a \emph{quis}, by the induction hypothesis. So we have $\omega_i$ such that $\delta_is_n e = \partial\omega_i$, for $i=1, \dots , n$. These $\partial \omega_i$ verify the Kan condition \ref{Kan-likecondition}, hence we get some $\partial \omega$ such that $\delta_i\partial\omega = \partial\omega_i, i =1, \dots , n$, because of \ref{enhancedKan-likeresult}. Therefore, we start with some arbitrary section $s_n$  obtained in the non-unitary case and rectify it with this $\partial\omega$: $s'_ne = s_ne - \partial\omega$. Then we perform the fancy average procedure to get a $\Sigma_n$-equivariant $\omega$ and check that everything works as it is supposed to work.
\end{proof}

\begin{remark}\label{unitsremainforever} Notice that the action of the units on $(P_+)_\infty$ we have built

$$
\delta_i = \_ \circ_i 1 : (P_+)_\infty(n) \otimes (P_+)_\infty(0) \longrightarrow (P_+)_\infty(n-1)\ , n > 1 \ , i=1, \dots , n
$$

reduces to the following cases:

\begin{itemize}
  \item[(a)] For $n=1$, it is the isomorphism $\mk \otimes \mk \longrightarrow \mk $.
  \item[(b)] For $n =2$, it is just the induced action from $P_+$: $HP_+(2) \otimes HP_+(0) \longrightarrow HP_+(1)$
  \item[(c)] For $n>2$, it is the trivial action $\omega \mapsto \omega\circ_i 1 = 0$.
\end{itemize}
\end{remark}

\subsection{Uniqueness. The non-unitary case} The following lemma provides a proof of the uniqueness up to isomorphism of  minimal models. It is inspired in \cite{HT90}, definition 8.3, and theorem 8.7. It also inspired a categorical definition of minimal objects: see \cite{Roi93} and \cite{Roi94c}, cf. \cite{GNPR10}.

\begin{lemma}\label{seccio} Let $P_\infty$ be a Sullivan minimal operad and $\rho : Q \longrightarrow P_\infty$ a \emph{quis} of operads. Then there is a section $\sigma : P_\infty \longrightarrow Q$, $\rho\sigma = \id_{P_\infty} $.
\end{lemma}

\begin{proof} We are going to build the section $\sigma : P_\infty \longrightarrow Q$ inductively on the arity:

$$
\xymatrix{
{I_0}\ar[r]\ar@{.>}[rrrrrd]_{\sigma_1} &{P_2}\ar[r]\ar@{.>}[rrrrd]^{\sigma_2}  &{\dots}\ar[r] &{P_n}\ar[r]\ar@{.>}[rrd]^{\sigma_n} &  {\dots}  \ar[r]  & {P_\infty} \ar@/^/@{.>}[d]^{\sigma}   \\
                 &  &  &  &  &  {Q} \ar[u]^{\rho}
}
$$

in such a way that:

\begin{itemize}\label{secciosigma}
\item[(1)] $\rho\sigma_n = \id_{P_n}$ (note that, because of the minimality, $\im \rho\sigma_n \subset P_n$), and
\item[(2)] ${\sigma_n}_{| P_{n-1}} = \sigma_{n-1}$.
\end{itemize}

So, we start with the universal morphism $\sigma_1 : I_0 \longrightarrow Q$ from the initial operad $I_0$ to $Q$. It's clear that $\rho\sigma_1 = \id_{I_0}$. Let us assume that we have already constructed up to $\sigma_{n-1} : P_{n-1} \longrightarrow Q$ satisfying conditions above $(1)$ and $(2)$ and let us define $\sigma_n : P_n \longrightarrow Q$ as follows: first, take the $\Sigma$-module $Q_{n-1} := \im (\sigma_{n-1}: P_{n-1} \longrightarrow Q)$. By the induction hypothesis, $\sigma_{n-1}$ is a monomorphism, so $\sigma_{n-1}: P_{n-1} \longrightarrow Q_{n-1}$ is an isomorphism of $\Sigma$-modules and $\rho_{\vert Q_{n-1}}$ its inverse.

Next, consider the following commutative diagram of $\Sigma_n$-modules,

$$
\begin{CD}
0 @>>> Q_{n-1}(n)     @>>>   Q(n)     @>>> Q(n)/Q_{n-1}(n)   @>>>  0 \\
@. @VV\rho(n)_{| Q_{n-1}} V @VV\rho (n) V @VV\overline{\rho} (n)V @.\\
0 @>>> P_{n-1}(n) @>>> P_\infty(n) @>>> P_\infty(n)/P_{n-1}(n) @>>> 0
\end{CD}  \ .
$$

in which the horizontal rows are exact. As we said, the first column is an isomorphism, and the second a {\it quis\/}. So the third column is also a {\it quis\/}. By minimality and lemma \ref{lamadredelcordero}, $P_\infty(n) = P_n(n) = P_{n-1} \oplus E(n) $. Hence, $P_\infty(n) / P_{n-1}(n) \cong E(n)$, with zero differential. So we have an epimorphism of $\Sigma_n$-modules, $Z(Q(n)/Q_{n-1}(n)) \longrightarrow H(Q(n)/Q_{n-1}(n)) \cong E(n)$. Take a section $s: E(n) \longrightarrow Z(Q(n)/Q_{n-1}(n))$ and consider the pull-back of $\Sigma_n$-modules

$$
\xymatrix{
{Q(n)}\ar@/^2pc/[drr]^{\rho(n)}\ar@/_2pc/[ddr]_{\pi(n)}\ar@{.>}[dr]^{(\pi(n)\ \rho(n)))}    &    &      \\
    & {(Q(n)/Q_{n-1}(n)) \times_{E(n)} P_\infty (n)} \ar[r]\ar[d]    &  {P_\infty(n)}  \ar[d]   \\
    & {Q(n)/Q(n-1)}  \ar[r]^{\overline{\rho(n)}}  &  {P_\infty(n)/P_{n-1}(n) \cong E(n)} \ .
}
$$

It is an easy exercise to check that the induced morphism $(\pi (n) \ \rho (n) )$ is an epimorphism.

Let $i: E(n) \hookrightarrow P_\infty(n)$ denote the inclusion. We can lift $(s \ i )$ in the diagram

$$
\xymatrix{
                        &  {Q(n)}\ar[d]^{(\pi(n)\ \rho(n))} \\
E(n) \ar@{.>}[ur]^{f} \ar[r]_-{(s\ i)} & {(Q(n)/Q_{n-1}(n)) \times_{E(n)} P_\infty(n)}
}
$$

to a morphism $f : E(n) \longrightarrow Q(n)$ such that $(\pi (n) \ \rho (n) ) \circ f = (s \ i)$. Note that here we are using bare projectiviness to lift morphisms, since we don not need them to commute with any differentials at this stage. Finally, define $\sigma_n : P_n \longrightarrow Q$ by $
\sigma_{n \vert P_{n-1}} = \sigma_{n-1}$ and $\sigma_{n \vert E(n)} = f $.

According to the universal property of principal extensions \ref{propuniversalHirsch}, in order to check that $\sigma_n$ is a morphism of operads, we only need to prove that $\sigma_{n-1}d_n e = \partial_{Q(n)} f e$ for every $e \in E(n)$. This is again an easy exercise, as it is to check that $\rho\sigma_n = \id_{P_n}$.
\end{proof}

\subsection{Uniqueness. The unitary case}

\begin{lemma}\label{seccio+} Let $P_\infty$ be a unitary Sullivan minimal operad and $\rho : Q \longrightarrow P_\infty$ a \emph{quis} of operads, where $Q$ has a unitary multiplication. Then, there is a unitary section $\sigma : P_\infty \longrightarrow Q$, $\rho\sigma = \id_{P_\infty} $.
\end{lemma}

\begin{proof} We follow the same strategy as in the previous results: starting with section $f$ already built for the non-unitary case,

\begin{equation}\label{eqseccio+}
\xymatrix{
                        &  {Q(n)}\ar[d]^{(\pi(n)\ \rho(n))} \\
E(n) \ar[ur]^{f} \ar[r]_-{(s\ i)} & {(Q(n)/Q_{n-1}(n)) \times_{E(n)} P_\infty(n) \ ,}
}
\end{equation}

we rectify it in order to make it compatible with the $\Lambda$-structures. Let us start with the case $n=2$. Diagram \ref{eqseccio+} reduces to

$$
\xymatrix{
                        &  {ZQ(2) \subset Q(2)}\ar[d]^{\id} \\
{E(2) \cong HQ(2)} \ar@{.>}[ur]^{f} \ar[r]_{s_2} & {ZQ(2) \subset Q(2) \ .}
}
$$

Hence $f=s_2$ and the problem boils down to proving that section $s_2$ can be chosen to be compatible with the $\Lambda$-structures. Again, this means studying the differences $\omega_i = \delta_is_2e - s_1\delta_ie \ , \quad \text{for} \ i =1,2$, for $e \in E(2)$. Here, $s_1 : HP_\infty (1) = \mk \longrightarrow ZQ(1)$ is the unique lifting of the isomorphism $\rho_*(1) : HP_\infty(1) \longrightarrow HQ(1)$ to $ZQ(1)$ that sends $\id \in HP_\infty (1)$ to $\id \in ZQ(1)$.

$$
\xymatrix{
{E(2)} \ar[r]^{s_2}\ar[d]^{\delta_i} & {ZQ(2)}  \ar[d]^{\delta_i}  \\
{HP_\infty(1) = HQ(1)} \ar[r]^-{s_1} &  {ZQ(1)\ .}
}
$$

Again, these $\omega_i$ don not need to be zero, but it is easy to verify that they are cocycles, satisfy the Kan-like condition \ref{Kan-likecondition} and belong to $\ker\rho$. Let's check this last statement: $  \rho\omega_i = \rho\delta_is_2e - \rho s_1\delta_i e = \delta_i\rho s_2 e - \rho s_1 \delta_i e = \delta_i e - \delta_i e = 0  $, for $i =1,2$. Here we have used the fact that $s_2$ and $s_1 $ are already sections of $\rho$, which we know from the non-unitary case, and that $\rho$ is a morphism of unitary operads.

Hence, because of lemma \ref{enhancedKan-likeresult} there is some $\omega \in ZQ(2) \cap \ker \rho$ such that $\delta_i \omega = \omega_i, i = 1,2$ and we use it to rectify our previous choice of a section: $ s'_2 e = s_2 e  - \omega$. Apply the average procedure to turn $s'$ into a $\Sigma_2$-equivariant morphism, and we are done with $n=2$.

Assume we have already built the section $\sigma$ up to the $n-1, n>2$, stage of the minimal operad $P_\infty$. Now the $\Lambda$-structure on $E(n)$ is the trivial one $\delta_i = 0, i=1, \dots n$. So we have to rectify our $f$ in \ref{eqseccio+} to verify $\delta_i f = 0, i = 0, \dots , n$. Again, for $e\in E(n)$ we don't necessary have $\omega_i = \delta_i fe = 0, i =0, \dots , n$. Nevertheless, $\pi \omega_i = \delta_i \pi fe = \delta_i se = s\delta_i e = 0$, because of the trivial $\Lambda$-structure on $E(n)$. Here we have used the fact that we could also have previously rectified also $s$ into a morphism of $\Lambda$-structures---a verification which we leave to the conscientious reader. Therefore, $\delta_ife \in Q_{n-1}(n-1), i =1, \dots , n$. Again these $\delta_ife$ verify the Kan condition \ref{Kan-likecondition}, and belong to $\ker\rho$. So, because of lemma \ref{enhancedKan-likeresult}, we get $\omega \in Q_{n-1}(n) \cap \ker\rho$ such that $\delta_i\omega = \omega_i$ and we redefine our $f$ from the non-unitary case as usual as $f'e = fe - \omega$. Again, we average to produce a $\Sigma_n$-equivariant $\widetilde{\omega}$ and check that everything works as it is supposed to.
\end{proof}

\subsection{Uniqueness: conclusions}

From lemmas \ref{seccio} and \ref{seccio+}, uniqueness follows at once from both the non-unitary and unitary cases.

\begin{proposition}\label{quisimplicaiso} Let $\rho : P_\infty \longrightarrow P'_\infty$ be a \emph{quis} between minimal Sullivan operads. Then, $\rho$ is an isomorphism.
\end{proposition}

\begin{proof} Because of the previous lemmas \ref{seccio} and \ref{seccio+}, $\rho$ has a section $\sigma$, which by the two out of three property is also a \emph{quis}. So $\sigma$ also has a section and it is both a monomorphism and an epimorphism.
\end{proof}

\begin{theorem}\label{unicitat} Let $\varphi :P_\infty \lra P$ and $\varphi': P'_\infty \lra P$ be two Sullivan minimal models of $P$. Then there is an isomorphism $\psi : P_\infty \lra P'_\infty$, unique up-to-homotopy, such that $\varphi' \psi\simeq \varphi$.
\end{theorem}

\begin{proof} The existence of $\psi$ follows from the up-to-homotopy lifting property \ref{strict_lifting}. It is a \emph{quis} because of the $2$-out-of-$3$ property and hence an isomorphism because of the previous proposition \ref{quisimplicaiso}.
\end{proof}

\section{Miscellanea}

In this section, we develop some corollaries relating the minimal models of $P_+$ and $P$, and establishing their relationship with up-to-homotopy algebras. Namely:

\begin{itemize}
  \item[(6.1)] We compare the minimal models of a unitary operad $P_+$ and its non-unitary truncation $P$.
  \item[(6.2)] We relate the minimal model of a unitary operad $(P_+)_\infty$ and the up-to-homotopy $P_+$-algebras with strict units.
  \item[(6.3)] For the case of the unitary associative operad, we compare our minimal model $su\Ass_\infty = \Ass_{+\infty}$ with the one of up-to-homotopy algebras with up-to-homotopy units, $hu\Ass_\infty$.
  \item[(6.4)] Here we extend some results of our previous paper \cite{CR19}, that we could not address there for the lack of minimal models for unitary operads.
  \item[(6.5)] We complete the results of \cite{GNPR05} concerning the formality of operads, so as to include unitary operads.
\end{itemize}

\subsection{Minimal models of an operad and its unitary extension} Let $P$ be an operad admitting a unitary extension $P_+$. We clearly have a split exact sequence of $\Sigma$-modules

$$
0 \longrightarrow P \longrightarrow P_+ \longrightarrow \mk [1] \longrightarrow 0 \ .
$$

Here, $P \longrightarrow P_+$ is the canonical embedding, $\mk [1] = \mk$ the $\Sigma$-module which is just a $\mk$-vector space on one generator $1$ in arity-degree $(0,0)$ and zero outside and $P_+ \longrightarrow \mk [1]$ the projection of $\Sigma$-modules that sends $P(l)$, $ l > 0$ to zero and the identity on $P_+(0) = \mk$. We could also write $P_+ = P \oplus \mk [1]$ as $\Sigma$-modules.

\begin{proposition}\label{twominimalmodels} For every cohomologically unitary, cohomologically connected and with unitary multiplication operad $P_+$ we have an isomorphism of operads $(P_+)_\infty = (P_\infty)_+ $.
\end{proposition}

In particular, we have an isomorphism of $\Sigma$-modules $(P_+)_\infty = P_\infty \oplus \mk [1]$.

\begin{proof} For every $P_+$, its truncation $P$ is a $\Lambda$-operad. and this structure passes to its minimal model $P_\infty$ (see remark \ref{unitsremainforever}). So, we have a unitary extension $(P_\infty)_+$. Let us see how this unitary extension agrees with the minimal model of $P_+$. Indeed, $P_\infty$ is a colimit of principal extensions

$$
I_0 \longrightarrow P_2 = \Gamma \left( E(2) \right) \longrightarrow \dots \longrightarrow P_n = \Gamma \left(\bigoplus_{n\geq 2} E(n)\right)\longrightarrow \dots \longrightarrow \dirlim_n P_n = P_\infty
$$

starting with the non-unitary initial operad $I_0$. For the same reasons we just remarked about $P_\infty$, all these operads $P_n$ have unitary extensions. So we can take the unitary extension of the whole sequence

$$
I_+ \longrightarrow (P_2)_+ = \Gamma (E(2))_+  \longrightarrow \dots \longrightarrow (P_n)_+ = \Gamma \left( \bigoplus_{n\geq 2} E(n) \right)_+    \longrightarrow \dots \longrightarrow (\dirlim_n P_n)_+ = (P_\infty)_+ \ .
$$

But, as we noticed in lemma \ref{extensionfunctor}, the functor $(\ )_+$ commutes with colimits, so $(P_\infty)_+ = (\dirlim_n P_n)_+ = \dirlim_n (P_{n+}) = (P_+)_\infty$.
\end{proof}

\subsection{Minimal models and up-to-homotopy algebras}

In the non-unitary case, the importance of these minimal models $P_\infty$ is well-known: they provide a \emph{strictification} of up-to-homotopy $P$-algebras. That is, up-to-homotopy $P$-algebras are the same as regular, \emph{strict} $P_\infty$-algebras. One way to prove it is the following: first, we have a commonly accepted definition for up-to-homotopy $P$-algebras, at least for \emph{Koszul} operads. Namely, the one in \cite{GK94}:

\begin{definition}(\cite{GK94}, see also \cite{LoVa12}) Let $P$ be a Koszul operad. Then an \emph{up-to-homotopy} $P$-algebra is an algebra over the Koszul resolution (model) $\Omega P^{\text{\text{!`}}}$ of $P$.
\end{definition}

So one proves that $\Omega P^{\text{!`}} \stackrel{\sim}{\longrightarrow} P$ is a minimal model of $P$, in the sense that it is unique up to isomorphism (\cite{LoVa12}, corollary 7.4.3). Since  Markl's minimal model \emph{à la Sullivan} $P_\infty \stackrel{\sim}{\longrightarrow} P$ is also minimal and cofibrant, we necessarily have an isomorphism $P_\infty = \Omega P^{\text{!`}}$ (see \cite{Mar96}). Next, one has to check that this definition as $\Omega P^{\text{!`}}$-algebras also agrees with the definitions through \lq\lq equations" in the particular cases. For instance, one has to check that $\Ass_\infty = \Omega\Ass^{\text{!`}}$-algebras are the same as $A_\infty$-algebras, defined as dg modules, together with a sequence of $n$-ary operations $\mu_n : A^{\otimes n} \longrightarrow A ,\  n \geq 2 ,\  |\mu_n| = 2 - n $, satisfying the equations $\partial (\mu_n) = \sum_{\substack{ p+q+r = n  \\  p+1+r=m}} (-1)^{qr+p+1} \mu_m \circ_{p+1} \mu_q $. (See \cite{LoVa12}, lemma 9.2.1.)

We would like to say that the same is true in the unitary case, in other words, to prove a theorem such as

\begin{theorem} Up-to-homotopy $P_+$-algebras with strict unit are the same as $(P_+)_\infty$-algebras.
\end{theorem}

However, one important ingredient is missing: we lack a common, accepted definition for up-to-homotopy $P_+$-algebras with strict units. To the best of our knowledge, such a definition exists only for the operad $u\Ass = \Ass_+ $. For instance, the one in \cite{KS09}, definition 4.1.1 (cf. \cite{Lyu11}, \cite{Bur18}):

\begin{definition} An $A_\infty$-algebra $A$ is said to have a \emph{strict unit} if there is an element $1\in A$ of degree zero such that $\mu_2(1,a) = a = \mu_2(a,1)$ and $\mu_n(a_1, \dots , 1, \dots, a_n) = 0$ for all $n\neq 2$ and $a, a_1, \dots , a_n \in A$.
\end{definition}

So we prove our theorem for the only case currently available: $P_+ = \Ass_+$.

\begin{theorem} $A_\infty$-algebras with strict unit are the same as $\Ass_{+\infty}$-algebras.
\end{theorem}

\begin{proof} We just have to prove that the unit $1 \in \Ass_{+\infty}(0)$ acts as described in the definition.  Namely, $\mu_2 \circ_1 1 = \id = \mu_2 \circ_2 1$ and $\mu_n \circ_i 1 = 0$, for $
n > 2$, and $i = 1, \dots , n$. As for the first equations, because of remark \ref{unitsremainforever}, partial composition products

$$
\circ_i : \Ass_{+\infty} (2) \otimes \Ass_{+\infty} (0) \longrightarrow \Ass_{+\infty} (1) \ , \quad i = 1, 2 \ ,
$$

are induced by $\circ_i : \Ass_+(2) \otimes \Ass_+(0) \longrightarrow \Ass_+ (1) \ , \quad i = 1, 2
$, which verify said identities. As for the rest of the equations, for $n>2$, again because of remark \ref{unitsremainforever}, partial composition products $\circ_i : \Ass_{+\infty} (n) \otimes \Ass_{+\infty} (0) \longrightarrow \Ass_{+\infty} (n-1)$ are trivial.
\end{proof}

\subsection{Strict units and up-to-homotopy units}

Here we compare two minimal models of a unitary operad $P_+$: the one with \emph{strict units} that we developed $(P_+)_\infty$, and the one with \emph{up-to-homotopy units} that we find for the case of the unitary associative operad in \cite{HM12} or \cite{Lyu11}. We use the notations $su\Ass_\infty = \Ass_{+\infty}$ and $hu\Ass_\infty$, respectively.

As \cite{HM12} mentions, $su\Ass_\infty$ \emph{cannot} be cofibrant, nor minimal, since if it were, we would have two \emph{quis} $su\Ass_\infty \stackrel{\sim}{\longrightarrow} u\Ass \stackrel{\sim}{\longleftarrow}   hu\Ass_\infty $ and hence, by the up-to-homotopy lifting property and the fact that both are minimal, we would conclude that both operads $su\Ass_\infty$ and $hu\Ass_\infty$ should be isomorphic, which we know they are not, just by looking at their presentations. We are suppressing the generators given by the actions of the permutations groups in order to lighten the formulas

$$
hu\Ass_\infty = \Gamma ( \{\mu_n^S\}_{S,n\geq2}) \ ,
$$

(see \cite{HM12}, \cite{Lyu11}) and

$$
su\Ass_\infty = \Gamma_+ (\{\mu_n\}_{n\geq 2}) = \dfrac{\Gamma  (1, \{\mu_n\}_{n\geq 2})}{\langle \mu_2 \circ_1 1 - \id,\ \mu_2 \circ_2 1 - \id, \ \{\mu_n \circ_i 1\}_{n\geq 2, i=1,\dots, n}\rangle }
$$

Nevertheless, we have indeed proven that $su\Ass_\infty$ is a minimal and cofibrant operad. And it is of course, but \emph{as an operad with unitary multiplication}, in $\Mag_+ \backslash \Op_+$. Even though \emph{it is not as an operad} in $\Op$. Indeed, looking at its second presentation, with the free operad functor $\Gamma$, we see that it seems to lack the first condition of minimality, i.e., being free as a graded operad. Again, there is no contradiction at all: it is free graded \emph{as a unitary operad}; that is, in $\Mag_+ \backslash \Op$, with the free operad functor $\Gamma_+$.

\begin{example}The \emph{free}  $\C$-algebra $CX$ in \cite{May72}, construction 2.4, and lemma 2.9, or, more generally, the \emph{free reduced} $P_+$-algebra $\mathbb{S}_*(P_+, X)$ for a unitary operad $P_+$ in \cite{Fre17a}, p. 74, are also examples of free objects when you consider them in categories of unitary algebras, but losing its \lq\lq freedom" in the categories of all algebras, unitary or otherwise.
\end{example}

This difference between the same object being free in a subcategory and not being free in a larger category has as a consequence that minimal objects in a subcategory can lose its minimality in the broader one. This paradoxical phenomenon, it is not so new and has already been observed (see, for instance, \cite{Roi94c}, remark 4.8). Here we present another example of it, but in the category of dg commutative algebras.

\begin{example}\label{nounitexample}  Let $\Cdga{\mathbb{Q}}$ denote the category of dg commutative algebras, \emph{without unit}. Let $\Cdga{\mathbb{Q}}_1$ denote the category of algebras \emph{with unit}. By forgetting the unit, we can consider $\Cdga{\mathbb{Q}}_1$ as a subcategory of $\Cdga{\mathbb{Q}}$.

$\mathbb{Q}$, being the initial object in $\Cdga{\mathbb{Q}}_1$, is free, cofibrant, and minimal  in $\Cdga{\mathbb{Q}}_1$. Indeed, if we denote by $\Lambda_1$ the free graded commutative algebra \emph{with unit} functor, then $\Lambda_1(0) = \mathbb{Q}$: the free graded commutative algebra \emph{with} unit on the $\mathbb{Q}$-vector space $0$. However, it is \emph{neither} minimal nor cofibrant, nor free as an object in the larger category $\Cdga{\mathbb{Q}}$. To see this, let us denote by $\Lambda$ the free graded-commutative algebra \emph{without} unit. As an algebra \emph{without} unit, $\mathbb{Q}$ has a new relation. Namely, $1^2 = 1$. So, it is \emph{not} a free algebra in $\Cdga{\mathbb{Q}}$: $\mathbb{Q} = \Lambda_1 (0) = \Lambda (1)/(1^2 - 1)$.

Next, consider the free graded-commutative algebra \emph{without} unit  $\Lambda (t,x)$ on two generators $t,x$ in degrees $|t|=-1$ and $|x|=0$ and differential $dx = 0$ and $dt = x^2 - x$. Hence, as a graded vector space,

$$
\Lambda (t,x)^{i} =
\begin{cases}
  (x), & \mbox{if } i = 0, \\
  [t, tx, tx^2, \dots , tx^n, \dots], & \mbox{if } i = -1, \\
  0 , & \mbox{otherwise}.
\end{cases}
$$

where:

\begin{itemize}
  \item[(1)] $(x)$ is the ideal generated by $x$ in the polynomial algebra $\mathbb{Q}[x]$. That is, the $\mathbb{Q}$-vector space $[x, x^2, \dots , x^n, \dots]$
  \item[(2)] $[t, tx, tx^2, \dots , tx^n, \dots]$ means the $\mathbb{Q}$-vector space generated by those vectors.
\end{itemize}

Consider the morphism of algebras \emph{without} unit $\varphi : \Lambda (t,x) \longrightarrow \mathbb{Q}$ defined by $\varphi(x) = 1, \varphi (t) = 0$. $\varphi$ is a \emph{quis} and an epimorphism. So, if $\mathbb{Q}$ were a minimal and  cofibrant algebra \emph{without unit}, we would have a section $\sigma : \mathbb{Q} \longrightarrow \Lambda (t,x)$, $\varphi \sigma = \id $. For degree reasons, we would then have $\sigma (1) = p(x)$, for some polynomial $p(x) \in (x)$. That is to say, a polynomial of degree $\geq 1$. But, since $\sigma(1)\sigma(1) = \sigma (1^2) = \sigma (1) $,  we would get $p(x)^2 = p(x) $, which is impossible for a polynomial of degree $\geq 1$.

Hence, $\mathbb{Q}$ is graded-free, cofibrant and minimal \emph{as algebra with unit}. But it is neither of those things \emph{as an algebra without unit}.
\end{example}

\subsection{Minimal models of operad algebras for tame operads} In \cite{CR19} we proved the existence and uniqueness of Sullivan minimal models for operad algebras, for a wide class of operads we called \lq\lq tame", and for operad algebras satisfying  just the usual connectivity hypotheses.

Of particular importance was the fact that, if an operad $P$ is tame, then its minimal model $P_\infty$ is also tame: that is, $P_\infty$-algebras also have Sullivan minimal models \cite{CR19}, proposition 4.10. This provides minimal models for $\Ass_\infty, \Com_\infty$ and $\Lie_\infty$-algebras, for instance. Since at that time we were not aware of the possibility of building minimal models for unitary operads, there was a gap in our statements, meaning we had to formulate them only for non-unitary operads (there called \lq\lq reduced"). Now we can mend that gap.

\begin{proposition} Let $P\in \Op$ be a cohomologically connected, cohomologically unitary and with unitary multiplication $r$-tame operad. Then its minimal model is also $r$-tame.
\end{proposition}

\begin{proof} Indeed, the presence of a non-trivial arity zero $P(0)$ adds nothing to the condition of being tame or not.
\end{proof}

\begin{corollary} Every cohomologically connected $\Ass_{+\infty}$ or $\Com_{+\infty}$-algebra has a Sullivan minimal model. Also, every $1$-connected $\Ger_{+\infty}$-algebra has a Sullivan minimal model.
\end{corollary}

Then we went on to prove the same results for pairs $(P,A)$, where $P$ is a tame operad and $A$ a $P$-algebra, thus providing a global invariance for our minimal models in the form of a minimal model $(P_\infty, \M) \stackrel{\sim}{\longrightarrow} (P,A)$ in the category of such pairs, the category of operad algebras \emph{over variable operads}. We can add now unitary operads to that result too.

\begin{theorem} Let $P$ be a cohomologically connected, cohomologically unitary, and with unitary multiplication $r$-tame operad and $A$ an $r$-connected $P$-algebra. Then $(P,A)$ has a Sullivan $r$-minimal model $(P_\infty, \M) \stackrel{\sim}{\longrightarrow} (P,A)$.
\end{theorem}

\subsection{Formality} It has been pointed out by Willwacher in his speech at the 2018 Rio's International Congress of Mathematicians, \cite{Wil18}, talking about the history of the formality of the little disks operads, that our paper \cite{GNPR05} missed the arity zero. Here we complete the results of that paper for the unitary case.

\begin{proposition}\label{formality+} Let $P_+$ be a unitary dg operad with $HP_+(0) = HP_+(1) = \mk$. Then $P_+$ is formal if and only if $P$ is formal.
\end{proposition}

\begin{proof} Since the truncation functor is exact, implication $\Longrightarrow$ is clear. In the opposite direction, because of the hypotheses, $P$ and $P_+$ have minimal models $P_\infty$ and $(P_\infty)_+ = (P_+)_\infty$. Assume $P$ is formal. Then we have a couple of \emph{quis} $HP \stackrel{\sim}{\longleftarrow} P_\infty \stackrel{\sim}{\longrightarrow} P$. Applying the unitary extension functor to this diagram, and taking into account that it is an exact functor because of \ref{extensionfunctor}, we get $(HP)_+ \stackrel{\sim}{\longleftarrow} (P_\infty)_+ \stackrel{\sim}{\longrightarrow} P_+$. Which, it is just $H(P_+) \stackrel{\sim}{\longleftarrow} (P_+)_\infty \stackrel{\sim}{\longrightarrow} P_+ $. Hence, $P_+$ is also a formal operad.
\end{proof}

\begin{corollary} (cf. \cite{Kon99}, \cite{Tam03}, \cite{GNPR05}, \cite{LaVo14}, \cite{FW18}) The unitary $n$-little disks operad $\D_{n+}$ is formal over $\mathbb{Q}$.
\end{corollary}

\begin{proof} Follows from \cite{GNPR05}, corollary 6.3.3 and our previous proposition \ref{formality+}.
\end{proof}

We can also offer a unitary version of the main theorem 6.2.1 in \emph{op.cit.} about the independence of formality from the ground field.

\begin{corollary} (cf. \cite{Sul77}, \cite{HS79}, \cite{Roi94b}, \cite{GNPR05}) Let $\mk$ be a field of characteristic zero, and let $\mk \subset \mathbf{K}$ be a field extension. Let $P$ be a cohomologically connected, cohomologically unitary and with unitary multiplication $\mk$-operad with finite type cohomology. Then $P$ is formal if and only if $P\otimes \mathbf{K}$ is formal.
\end{corollary}

\begin{proof} Because the statements only depend on the homotopy type of the operad, we can assume $P$ to be minimal, and hence connected and unitary: let us call it $P_+$. Then, $P_+$ is formal if and only if its truncation $P$ is so because of the previous proposition \ref{formality+}. Because of op.cit. theorem 6.2.1, $P$ is formal if and only if $P\otimes \mathbf{K}$ is so. Because of previous proposition \ref{formality+}, this is true if and only if $(P\otimes \mathbf{K})_+ = P_+ \otimes \mathbf{K}$ is formal.
\end{proof}

The interested reader can easily check that the rest of the sections of \cite{GNPR05} concerning non-unitary operads admit similar extensions to unitary ones. This is true, even for the finite type results like theorem 4.6.3 in \emph{op.cit} from which the descent of formality hinges:

\begin{theorem} Let $P$ be a cohomologically connected, cohomologically unitary and with unitary multiplication operad. If the cohomology of $P$ is of finite type, then its minimal model $P_\infty$ is of finite type.
\end{theorem}

And this is so because, even in the unitary case, $P_\infty = \Gamma (E)$, with $E(0) = E(1) =0$.

In particular, we have the celebrated Sullivan's criterium of formality based on the lifting of a \emph{grading automorphism} also for unitary operads.

\begin{definition} Let $\alpha \in \mk^*$ to not be a root of unity and $C$  a complex of $\mk$-vector spaces. The \emph{grading automorphism} $\phi_\alpha$ of $HC$ is defined by $\phi_\alpha = \alpha^i \id_{HC^i}$ for all $i\in \mathbb{Z}$. A morphism of complexes $f$ of $C$ is said to be a \emph{lifting} of the grading automorphism if $Hf = \phi_\alpha$.
\end{definition}

\begin{proposition}(cf. \cite{Sul77}, \cite{GNPR05}) Let $P$ be a cohomologically connected, cohomologically unitary and with unitary multiplication operad with finite type cohomology. If for some non-root of unity $\alpha \in \mk^*$, $P$ has a lifting of $\phi_\alpha$, then $P$ is formal.
\end{proposition}

\appendix
\section{The Kan-like structure of an operad with unitary multiplication}

Let $P$ be an operad with a unit $1 \in P_+(0)$. Then we can define \emph{face maps}

$$
\delta_i : P(n) \longrightarrow P(n-1) \ , \quad i = 1, \dots , n \ , \ n \geq 1 \ , \qquad
\delta_i (\omega ) = \omega \circ_i 1 \ .
$$

Equivalently, start with a $\Lambda$-operad $P$, with its restriction operations $\delta_i : P(n) \longrightarrow P(n-1)$ subjected to verify the axioms of \cite{Fre17a}, p. 70-71.

Let $P$ be an operad \emph{with unitary multiplication}. That is, we have a morphism of operads $\varphi : \Mag_+ \longrightarrow P$. Let us call $m = \varphi (\mu)$ the image of the arity two generator $\mu \in \Mag_+(2)$ and $1 \in P(0)$ the image of $1 \in \Ass_+(0)$. This means we have an associative multiplication with unit $m \in P(2)$, which is a cocycle. That is, $m \circ_1 1 = \id = m \circ_2 1$, and $\partial m = 0 = \partial 1$. Equivalently, the first equations can be written in terms of the $\Lambda$-structure as $\delta_1 m = \id = \delta_2 m $.

This multiplication allows us to define \emph{degeneracy maps}

$$
s_i : P(n) \longrightarrow P(n+1) \ , \quad i = 1, \dots , n \ , \ n \geq 0 \ ,
\qquad
s_i(\omega ) = \omega \circ_i m \ .
$$

\begin{remark}\label{differentials} Notice that both \emph{face} and \emph{degeneracy} maps commute with the differentials of the operad.
\end{remark}

\begin{lemma}\label{simplicial} Let $P$ be an operad with unitary multiplication. Then the face and degeneracy maps verify:
\begin{enumerate}
  \item[(i)] $\delta_i\delta_j = \delta_{j-1}\delta_i$, if $i<j$,
  \item[(ii)] (a) $\delta_is_j = s_{j-1}\delta_i$, if $i<j$, and (b) $\delta_js_j = \id = \delta_{j+1}s_j$
\end{enumerate}
\end{lemma}

\begin{proof} Assertion (i) follows from the functorial relation $(\delta^j\delta^i)^* = \delta_i\delta_j$, together with the explicit formulas for the increasing maps $\delta^i : \underline{n-1} \longrightarrow \underline{n}$ in \cite{Fre17a}, p. 58-59. Assertion (iia) follows from inspection. Assertion (iib) follows from the fact that $m$ is a unitary multiplication.
\end{proof}

So, up to a shift of the arity index, every unitary operad with multiplication has the structure of an augmented simplicial-like complex. We also have a Kan-like condition and Kan-like property in unitary operads.

\begin{definition}\label{Kan-likecondition} Let $\left\{\omega_i \right\}_{i=1, \dots , n}$ be a family of elements in $P(n-1)$. We say that these elements verify a \emph{Kan-like condition} if $
\delta_i \omega_j = \delta_{j-1}\omega_i$, for all $ i < j $.
\end{definition}

\begin{example} Elements $\omega \in P(n), n \geq 1 ,$ produce families $\left\{ \omega_i = \delta_i \omega\right\}_{i = 1, \dots , n}$ in $P(n-1)$ that verify the \emph{Kan-like condition}. This follows from the first identity in \ref{simplicial}.
\end{example}

The main result in this appendix says the reciprocal of this example is also true.

\begin{lemma}\label{Kan-likeresult} Let $\left\{ \omega_i  \right\}_{i = 1, \dots , n}$ be a family of elements in $P(n-1)$ verifying the Kan-like condition. Then there exists $\omega \in P(n)$ such that $\delta_i \omega = \omega_i$ for all $i = 1, \dots , n$.
\end{lemma}

\begin{proof} We can easily adapt the proof in \cite{May92}, theorem 17.1. Precisely, we are going to inductively construct elements $u_1, \dots , u_n \in P(n)$ such that $\delta_i u_r = \omega_i$ for all $ i \leq r$ and $ r = 1 , \dots , n $. Then we will just take $\omega = u_n$. We start with $u_1 = s_1\omega_1$. We have $\delta_1u_1 = \delta_1s_1\omega_1 = \omega_1$, because of \ref{simplicial}, (iiib). Assume $u_1, \dots , u_{r-1}$ have already been constructed. Then, define $u_r$ as follows:

$$
  y_{r-1} = s_r(-\delta_ru_{r-1} + \omega_r)  \ ,
  \qquad
  u_r = u_{r-1} + y_{r-1} \ .
$$

Let us verify that $\delta_i y_{r-1} = 0$, for $i<r$:

\begin{eqnarray*}
  \delta_iy_{r-1} &=& \delta_is_r(-\delta_ru_{r-1} + \omega_r ) = s_{r-1}\delta_i ( -\delta_ru_{r-1} + \omega_r  ) \ , \quad \text{because \ref{simplicial}, (iiia)}, \\
   &=& -s_{r-1}\delta_i\delta_ru_{r-1} + s_{r-1}\delta_i\omega_r = -s_{r-1}\delta_{r-1}\delta_i u_{r-1} + s_{r-1}\delta_i\omega_r \ , \quad \text{because \ref{simplicial}, (i)},\\
   &=& -s_{r-1}\delta_{r-1}\omega_i + s_{r-1}\delta_i\omega_r \ , \quad \text{by the induction hypothesis}, \\
   &=& -s_{r-1}\delta_{r-1}\omega_i + s_{r-1}\delta_{r-1}\omega_i = 0  \quad \text{because of the Kan-like condition} \ .
\end{eqnarray*}

Hence, by the induction hypothesis, for all $i< r$, we have $\delta_iu_r = \delta_iu_{r-1}  + \delta_iy_{r-1} = \delta_iu_{r-1} = \omega_i$. Finally, because \ref{simplicial} $\delta_r u_r = \delta_r u_{r-1} + \delta_rs_r (-\delta_r u_{r-1} + \omega_r) = \delta_ru_{r-1} - \delta_ru_{r-1} + \omega_r  = \omega_r $
\end{proof}

\begin{remark} Notice the operations we perform in order to produce the element $\omega$ such that $\delta_i \omega = \omega_i$: $\delta_i, s_i$, and additions. Hence, we can refine our statement: if the elements $\omega_i$ belong to a certain submodule and subsimplicial-like complex $B \subset P$, then also $\omega \in B$.
\end{remark}

\begin{examples}
\begin{itemize}
  \item[(1)] If $B \subset P$ is a dg ideal, then it is a sub-dg-module and a sub-simplical-like complex. For instance, the cocycles of the operad $B = ZP$, or the kernel of a morphism $\rho: P \longrightarrow Q$ of operads $B = \ker\rho$, are submodules and subsimplicial-like complexes.
  \item[(2)] The boundaries of the operad $B = BP$ and the image of a morphism of operads $B = \im\rho$ are sub-dg-modules and subsimplicial complexes.
\end{itemize}
\end{examples}

\begin{lemma}\label{enhancedKan-likeresult} Let $B \subset P$ be a sub-dg-module and a sub-simplicial complex of $P$. Let $\left\{ \omega_i  \right\}_{i = 1, \dots , n}$ be a family of elements in $B(n-1)$ verifying the Kan-like condition. Then, there exists an element $\omega \in B(n)$ such that $\delta_i \omega = \omega_i$ for all $i = 1, \dots , n$.
\end{lemma}

\subsubsection*{Acknowledgments} I thank Daniel Tanré for pointing to the book \cite{Fre17a} and \cite{Fre17b}, which has been essential for the writing of this paper, and also for his pertinent questions about example \ref{nounitexample}, which forced me to rewrite it in order to make it more readable than the original versions. Joana Cirici taught me how to streamline the basic homotopy theory of almost everything while writing our \cite{CR19}, which has been helpful in section 3 of this paper. Fernando Muro pointed out that in the first appearance of up-to-homotopy things, they already had \emph{strict} units. I'm also grateful to Andrew Tonks and Imma Gálvez for listening to my explanations of some early draft versions of this paper.

\linespread{1}
\bibliographystyle{amsalpha}
\bibliography{bibliografia}

\providecommand{\bysame}{\leavevmode\hbox to3em{\hrulefill}\thinspace}
\providecommand{\MR}{\relax\ifhmode\unskip\space\fi MR }
\providecommand{\MRhref}[2]{%
  \href{http://www.ams.org/mathscinet-getitem?mr=#1}{#2}
}
\providecommand{\href}[2]{#2}
\begin{thebibliography}{FOOO09b}

\bibitem[Bur18]{Bur18}
J.~Burke, \emph{Strictly unital {$\rm A_\infty$}-algebras}, J. Pure Appl.
  Algebra \textbf{222} (2018), no.~12, 4099--4125.

\bibitem[BV73]{BoVo73}
J.~M. Boardman and R.~M. Vogt, \emph{Homotopy invariant algebraic structures on
  topological spaces}, Lecture Notes in Mathematics, Vol. 347, Springer-Verlag,
  Berlin-New York, 1973.

\bibitem[CR19]{CR19}
J.~Cirici and A.~Roig, \emph{Sullivan minimal models of operad algebras},
  Pub.Mat. \textbf{63} (2019), no.~1, 125–154.

\bibitem[FOOO09a]{FOOO09a}
K.~Fukaya, Y.-G. Oh, H.~Ohta, and K.~Ono, \emph{Lagrangian intersection {F}loer
  theory: anomaly and obstruction. {P}art {I}}, AMS/IP Studies in Advanced
  Mathematics, vol.~46, American Mathematical Society, Providence, RI;
  International Press, Somerville, MA, 2009.

\bibitem[FOOO09b]{FOOO09b}
\bysame, \emph{Lagrangian intersection {F}loer theory: anomaly and obstruction.
  {P}art {II}}, AMS/IP Studies in Advanced Mathematics, vol.~46, American
  Mathematical Society, Providence, RI; International Press, Somerville, MA,
  2009.

\bibitem[Fre17a]{Fre17a}
B.~Fresse, \emph{Homotopy of operads and {G}rothendieck-{T}eichmüller groups.
  {P}art 1. {T}he algebraic theory and its topological background},
  Mathematical Surveys and Monographs, vol. 217, American Mathematical Society,
  Providence, RI, 2017.

\bibitem[Fre17b]{Fre17b}
\bysame, \emph{Homotopy of operads and {G}rothendieck-{T}eichm\"{u}ller groups.
  {P}art 2. {T}he applications of (rational) homotopy theory methods},
  Mathematical Surveys and Monographs, vol. 217, American Mathematical Society,
  Providence, RI, 2017.

\bibitem[FTW18]{FTW18}
B.~Fresse, V.~Turchin, and T.~Willwacher, \emph{The homotopy theory of operad
  subcategories}, J. Homotopy Relat. Struct. \textbf{13} (2018), no.~4,
  689--702.

\bibitem[FW18]{FW18}
B.~Fresse and T.~Willwacher, \emph{The intrinsic formality of $e_n$-operads},
  preprint arXiv:1503.08699 (2018).

\bibitem[GK94]{GK94}
V.~Ginzburg and M.~Kapranov, \emph{Koszul duality for operads}, Duke Math. J.
  \textbf{76} (1994), no.~1, 203--272.

\bibitem[GM13]{GM13}
P.~Griffiths and J.~Morgan, \emph{Rational homotopy theory and differential
  forms}, second ed., Progress in Mathematics, vol.~16, Springer, New York,
  2013.

\bibitem[GNPR05]{GNPR05}
F.~Guill{\'e}n, V.~Navarro, P.~Pascual, and A.~Roig, \emph{Moduli spaces and
  formal operads}, Duke Math. J. \textbf{129} (2005), no.~2, 291--335.

\bibitem[GNPR10]{GNPR10}
\bysame, \emph{A {C}artan-{E}ilenberg approach to homotopical algebra}, J. Pure
  Appl. Algebra \textbf{214} (2010), no.~2, 140--164.

\bibitem[Hal83]{Hal83}
S.~Halperin, \emph{Lectures on minimal models}, M\'em. Soc. Math. France (N.S.)
  (1983), no.~9-10, 261.

\bibitem[Hin97]{Hin97}
V.~Hinich, \emph{Homological algebra of homotopy algebras}, Comm. Algebra
  \textbf{25} (1997), no.~10, 3291--3323.

\bibitem[Hin03]{Hin03}
\bysame, \emph{Erratum to \lq\lq{H}omological algebra of homotopy algebras"},
  preprint arXiv:0309453 (2003).

\bibitem[HM12]{HM12}
J.~Hirsh and J.~Mill\`es, \emph{Curved {K}oszul duality theory}, Math. Ann.
  \textbf{354} (2012), no.~4, 1465--1520.

\bibitem[HS79]{HS79}
S.~Halperin and J.~Stasheff, \emph{Obstructions to homotopy equivalences}, Adv.
  in Math. \textbf{32} (1979), no.~3, 233--279.

\bibitem[HT90]{HT90}
S.~Halperin and D.~Tanré, \emph{Homotopie filtrée et fibrés
  $\mathcal{C}^\infty$}, Illinois J. Math. \textbf{34} (1990), no.~2, 284--324.

\bibitem[KM95]{KrMay95}
I.~K{\v{r}}{\'{\i}}{\v{z}} and J.~P. May, \emph{Operads, algebras, modules and
  motives}, Ast\'erisque (1995), no.~233, iv+145pp.

\bibitem[Kon99]{Kon99}
M.~Kontsevich, \emph{Operads and motives in deformation quantization}, Lett.
  Math. Phys. \textbf{48} (1999), no.~1, 35--72, Mosh\'{e} Flato (1937--1998).

\bibitem[KS09]{KS09}
M.~Kontsevich and Y.~Soibelman, \emph{Notes on {$A_\infty$}-algebras,
  {$A_\infty$}-categories and non-commutative geometry}, Homological mirror
  symmetry, Lecture Notes in Phys., vol. 757, Springer, Berlin, 2009,
  pp.~153--219.

\bibitem[LV12]{LoVa12}
J.L. Loday and B.~Vallette, \emph{Algebraic operads}, Grundlehren der
  Mathematischen Wissenschaften [Fundamental Principles of Mathematical
  Sciences], vol. 346, Springer, Heidelberg, 2012.

\bibitem[LV14]{LaVo14}
P.~Lambrechts and I.~Voli\'{c}, \emph{Formality of the little {$N$}-disks
  operad}, Mem. Amer. Math. Soc. \textbf{230} (2014), no.~1079, viii+116.

\bibitem[Lyu11]{Lyu11}
V.~Lyubashenko, \emph{Homotopy unital {$A_\infty$}-algebras}, J. Algebra
  \textbf{329} (2011), 190--212.

\bibitem[Mar96]{Mar96}
M.~Markl, \emph{Models for operads}, Comm. Algebra \textbf{24} (1996),
  1471--1500.

\bibitem[Mar04]{Mar04}
\bysame, \emph{Homotopy algebras are homotopy algebras}, Forum Math.
  \textbf{16} (2004), no.~1, 129--160.

\bibitem[May72]{May72}
J.~P. May, \emph{The geometry of iterated loop spaces}, Springer-Verlag,
  Berlin-New York, 1972, Lectures Notes in Mathematics, Vol. 271.

\bibitem[May92]{May92}
J.~Peter May, \emph{Simplicial objects in algebraic topology}, Chicago Lectures
  in Mathematics, University of Chicago Press, Chicago, IL, 1992, Reprint of
  the 1967 original.

\bibitem[MSS02]{MSS02}
M.~Markl, S.~Shnider, and J.~Stasheff, \emph{Operads in algebra, topology and
  physics}, Mathematical Surveys and Monographs, vol.~96, American Mathematical
  Society, Providence, RI, 2002.

\bibitem[MT14]{MT14}
F.~Muro and A.~Tonks, \emph{Unital associahedra}, Forum Math. \textbf{26}
  (2014), no.~2, 593--620.

\bibitem[Mur16]{Mur16}
F.~Muro, \emph{Homotopy units in {$A$}-infinity algebras}, Trans. Amer. Math.
  Soc. \textbf{368} (2016), no.~3, 2145--2184.

\bibitem[Pos11]{Pos11}
L.~Positselski, \emph{Two kinds of derived categories, {K}oszul duality, and
  comodule-contramodule correspondence}, Mem. Amer. Math. Soc. \textbf{212}
  (2011), no.~996, vi+133. \MR{2830562}

\bibitem[Roi93]{Roi93}
A.~Roig, \emph{Minimal resolutions and other minimal models}, Publ. Mat.
  \textbf{37} (1993), no.~2, 285--303.

\bibitem[Roi94a]{Roi94b}
\bysame, \emph{Formalizability of dg modules and morphisms of cdg algebras},
  Illinois J. Math. \textbf{38} (1994), no.~3, 434--451.

\bibitem[Roi94b]{Roi94c}
\bysame, \emph{Mod\`eles minimaux et foncteurs d\'eriv\'es}, J. Pure Appl.
  Algebra \textbf{91} (1994), no.~1-3, 231--254.

\bibitem[Sta63]{Sta63}
J.~Stasheff, \emph{Homotopy associativity of {$H$}-spaces. {I}, {II}}, Trans.
  Amer. Math. Soc. 108 (1963), 275-292; ibid. \textbf{108} (1963), 293--312.

\bibitem[Sul77]{Sul77}
D.~Sullivan, \emph{Infinitesimal computations in topology}, Inst. Hautes
  \'Etudes Sci. Publ. Math. (1977), no.~47, 269--331 (1978).

\bibitem[Tam03]{Tam03}
D.~E. Tamarkin, \emph{Formality of chain operad of little discs}, Lett. Math.
  Phys. \textbf{66} (2003), no.~1-2, 65--72.

\bibitem[Wil18]{Wil18}
T.~Willwacher, \emph{Little disks operads and {F}eynman diagrams}, Proc. Int.
  Cong. Of Math., 2018.

\end{thebibliography}
\mbox{}\\
\linespread{1.2}

\end{document}